\documentclass[11pt]{amsart}
\input{preamble.tex}

\title{Higher idempotent completion for Soergel bimodules}
\author{Isabela Recio}
\address{Fachbereich Mathematik, Universit\"at Hamburg, 
Bundesstra{\ss}e 55, 
20146 Hamburg, Germany}
\email{isabela.recio.hernandez@uni-hamburg.de}

\begin{document}

\begin{abstract}
Higher idempotent completions have recently emerged as an important concept in condensed matter physics and topological quantum field theory. In this work, we present two applications of this concept to higher categories relevant in link homology theory and higher representation theory. We begin by studying singular Soergel bimodules, which, in Coxeter type $A$, have been identified---via the concept of perverse schobers---as the origin of the homotopy-coherent Rouquier braiding on chain complexes of Soergel bimodules. Our first main result establishes a converse connection for all Coxeter types: we show that singular Soergel bimodules can be recovered from Soergel bimodules through partial 2-categorical idempotent completions. Specializing to type $A$, we further assemble singular Soergel bimodules into a semistrict monoidal 2-category and identify certain quotients as semistrict monoidal 2-categories of progressive $\glN$-foams. Our second main result uses 2-categorical idempotent completions to formulate a higher categorical branching rule for such foam theories. These theories underlie the Lee--Gornik--Rasmussen--Wu deformations of coloured $\glN$ link homology. In particular, we provide a fully local version of Rose--Wedrich's decomposition theorem on deformed coloured link homology. 
  
\end{abstract}
\maketitle
\small\tableofcontents

\section{Introduction}

Categorical completion operations have become an indispensable tool in many areas of mathematics.  Describing a category of interest as the completion of a more manageable subcategory can drastically reduce its complexity, for example from finite-dimensional linear algebra to vector and matrix calculus. The classical \emph{idempotent completion} or \emph{Karoubi envelope} is the universal way of adjoining images of idempotent endomorphisms $p\colon c\to c$ to a category.

The framework of \emph{higher idempotent completion} \cite{DR19,GJF19, D22, RZ25} generalizes the Karoubi envelope to higher categories such as 2-categories, $n$-categories or their $\infty$-categorical analogs. In such settings the notion of equality of morphisms is weakened to a type of equivalence mediated by higher data. It is then natural to consider weaker notions of idempotents. For example, a \emph{2-categorical idempotent} or \emph{2-condensation monad} $p\colon c\to c$ does not need to satisfy $p\circ p=p$ on the nose. Instead one asks for a comparison 2-morphism $\mu \colon p\circ p \To p$, which does not even need to be invertible, as long as it admits a right-inverse. The appropriately defined \emph{image} or \emph{2-condensation} of a 2-categorical idempotent may or may not exist in the higher category of interest. In the latter case, the \emph{2-categorical idempotent completion} is the universal way of adjoining such data.

In this work, we explore applications of higher idempotent completion to higher categorical structures appearing in link homology theories and representation theory. 

\subsection{Higher idempotent completion of Soergel bimodules}
We begin with monoidal categories of \emph{Soergel bimodules}, which were developed as categorifications of Hecke algebras \cite{S92} and which have since become central objects of modern representation theory. Amongst other applications, complexes of Soergel bimodules yield categorical braid group actions \cite{R06} and form the underlying structure for link homology theories in the sense of Khovanov, see e.g. \cite{K07,LMGRSW24}. They are ultimately responsible for the sensitivity and usefulness of such theories in smooth 4-manifold topology \cite{R10, MWW22, RW24}.  

The (monoidal) categories of Soergel bimodules admit a generalization to 2-categories of \emph{singular Soergel bimodules} \cite{W11}, which categorify partial idempotent completions of Hecke algebras, a.k.a. \emph{Hecke algebroids}, give rise to coloured link homology theories \cite{MSV09, WW17}, and form prototypical examples of perverse schobers in type $A$ \cite{DW25}. 
 
Our first main result is that the procedure of 2-idempotent completion yields a formal description of singular Soergel bimodules through Soergel bimodules. We prove the following statement in the form of \Cref{thm:main} and \Cref{cor:main} for every Coxeter system under standard assumptions on the realization.

\begin{theoremA}
\label{thmA}
    The inclusion $\sbim\hookrightarrow{\sSbim}$ of Soergel bimodules into singular Soergel bimodules yields an equivalence of the 2-categorical idempotent completions
    \[
    \kkSBim \xrightarrow{\simeq} \kksSBim
    \]
so that the fully faithful 2-functor
\begin{equation}
    \label{eq:partialcompletion}
\sSbim \hookrightarrow \kksSBim \xrightarrow{\simeq} \kkSBim\end{equation}
exhibits singular Soergel bimodules as a \emph{partial 2-categorical idempotent completion} of Soergel bimodules.
\end{theoremA}
 Soergel bimodules have been made highly accessible via the Elias--Khovanov--Williamson diagrammatic calculus \cite{EKh10,EW16}, which has facilitated computational results such as \cite{EK10, L15, SW24}. An analogous combinatorial and graphical language is currently being developed for singular Soergel bimodules and has proven to be quite complex; see \cite{EKLP24_demazure, EKLP24_singular} and related work.

The 2-categorical idempotent completion provides a conceptual framework---in some sense: the universal setting---for an extension of the diagrammatic calculus to the singular case. In \Cref{sec:ssbimcompl} we explore the 2-categorical idempotent completion in type $A$ and pursue \Cref{q:idempotentcompletion} on whether \eqref{eq:partialcompletion} may be an equivalence. This would be equivalent to the 2-categorical idempotent completeness of $\sSbim$ and to $\sbim$ having $\sSbim$ as its full (not only partial) completion. Deciding these questions amounts to classifying separable Frobenius algebra objects in $\sbim$ and their Morita equivalence classes. We provide partial results and examples in small ranks and comment on connections to Kazhdan--Lusztig cell theory and the Klein--Elias--Hogancamp conjecture \cite{K14, EH17, MMMTZ23} on Frobenius algebra structures on indecomposable Soergel bimodules of distinguished involutions.

\subsection{Higher idempotent completion of foams} 
\label{sub:intro_higher}

For a second area of applications of the theory of $2$-categorical idempotent completion we consider the diagrammatic higher categories underlying the combinatorial description of the bigraded general linear link homology theories of Khovanov--Rozansky~\cite{KhR08}: the monoidal 2-category of progressive $\glN$-foams \cite{QR16}. 

Before describing the concrete application, we comment on the connection to the previous context. An important structural feature in Coxeter type $A$ is that parabolic subgroup inclusions endow not only the collection of symmetric groups with a monoidal structure, but also the collections of associated Hecke algebras, and the 2-categories of Bott-Samelson bimodules and Soergel bimodules. In \cite{SW24} (and in a different framework also in \cite{LMGRSW24}) the first step before the construction of a braiding is the construction of locally linear (semistrict) monoidal 2-categories $\mbsbim$ and $\msbim$ of all type $A$ Bott--Samelson bimodules and all type $A$ Soergel bimodules, respectively. In \Cref{prop:foambsbim} and \Cref{cor:2functors}, we extend these constructions to the singular cases and obtain locally linear (semistrict) monoidal 2-categories $\msbsbim$ and $\mssbim$.

\begin{theoremB}

\label{thmB}
    The progressive $\glN$-foam 2-categories of \cite{QR16} assemble into a locally linear semistrict monoidal 2-category $\NFoam_+$ that is realized as a quotient of $\msbsbim$ in the following commutative diagram of locally linear monoidal 2-functors
    \[ \begin{tikzcd}
	     \msbim & \mbsbim & \\
	      \mssbim & \msbsbim & 
           \NFoam_+ 
	       \arrow[from=1-2, to=2-3]
           \arrow["\kar",from=1-2, to=1-1]
           \arrow["\kar",from=2-2, to=2-1]
           \arrow[twoheadrightarrow, from=2-2, to=2-3]
	       \arrow[hookrightarrow,from=1-2, to=2-2]
           \arrow[hookrightarrow,from=1-1, to=2-1]
        %     \arrow[hookrightarrow, from=1-2, to=1-3]
        %    \arrow[hookrightarrow, from=2-2, to=1-3]
    \end{tikzcd}
    \]
    for locally $\C$-linear incarnations of these categories. 
\end{theoremB}

The coloured $\glN$ link homology theories associated with $\NFoam_+$ are sometimes described as $\mathrm{GL}(N)$-equivariant because the coloured unknot invariants, and thus the underlying Frobenius algebras, are $\mathrm{GL}(N)$-equivariant cohomology rings of Grassmannians. For example, the uncoloured unknot invariant is:

\[
\C[e_1,\dots,e_N, X]/\langle X^N-e_1 X^{N-1}+e_2 X^{N-2}- \cdots + (-1)^N e_N\rangle \cong H^*_{\mathrm{GL}(N)}(\C P^{N-1})
\]

To obtain the usual (coloured) $\glN$ link homology $\KhR^N$ \cite{KhR08, QR16} or one of its deformed versions $\KhR^\Sigma$ \cite{L05, G04, R10, W13, RW16}, one needs to take a quotient of $\NFoam_+$ by an additional foam relation, see \Cref{def:deformed_foams}, to obtain the deformed category $\defNFoam$. This relation amounts to specifying an $N$-element multiset $\Sigma$ of complex numbers as roots of the characteristic polynomial of the operator $X$ acting on the unknot invariant, which then becomes:

\[
\C[X]/\langle \prod_{\lambda\in \Sigma} (X-\lambda)\rangle 
=
\C[X]/\langle X^N-e_1(\Sigma) X^{N-1}+e_2(\Sigma) X^{N-2}- \cdots + (-1)^N e_N(\Sigma)\rangle
\]
Here the coefficients are signed elementary symmetric polynomials evaluated on $\Sigma$. 
Perhaps the most well-known version of this is Lee's deformation of Khovanov homology \cite{L05} for $\C[X]/\langle X^2 -1 \rangle$, which underlies the definition of Rasmussen's $s$-invariant \cite{R10}. The well-known computation of the dimension of the Lee homology of any $l$-component link as $2^l$ admits a beautiful explanation in terms of the Karoubi envelope of the Lee-version of Bar-Natan's dotted cobordism category \cite{BNM06}. One interpretation of the deformation behaviour is that every link component selects one of the two summands of the Levi subalgebra $\glnn{1}\oplus\glnn{1} \subset \glnn{2}$ \cite{W25}. The two copies are non-interacting and each lead to a trivial link invariant that is always $1$-dimensional. Summing all $2^l$ selections yields the total dimension.

The deformation behaviour of coloured $\glN$ homology depends on the multiplicities $N_i$ of the pairwise different elements $\lambda_i$ of $\Sigma$, i.e. on the partition $N=\sum_i N_i$, and resembles the branching rule for restricting $\glN$ representations to the Levi subalgebra $\bigoplus_{i} \glnn{N_i}$. More specifically, let $K$ be a knot and use $K^k$ to denote $K$ coloured with the $k$-th fundamental (exterior power) representation. Then \cite[Theorem 1.1]{RW16} specializes (from the generality of links with possibly several components) to yield an isomorphism of singly-graded vector spaces:

\begin{align}
\label{eq:krformula}
\KhR^{\Sigma}(K^k) \cong \bigotimes_{\sum_i k_i=k} \bigotimes_i \KhR^{N_i} (K^{k_i})    
\end{align}

Here the summands are canonically indexed by choices of multisubsets $A\subset \Sigma$ of size $|A|=k$ and $k_i$ denotes the multiplicity of $\lambda_i$ in $A$.

The 2-categorical idempotent completion allows us to extend the approach from \cite{BNM06} and \cite{RW16} by one level. Instead of just splitting webs and foams by idempotents, i.e. 1-morphisms and 2-morphisms, we can now also split objects in foam 2-categories, i.e. endpoints of tangles. The following appears as \Cref{thm:full_split} in the main text. 

\begin{theoremC}
\label{thmC}
Let $\Sigma$ be a multiset of pairwise distinct complex deformation parameters $\lambda_i$ with multiplicities $N_i$ and $N=\sum_i N_i$. Then in the 2-categorical idempotent completion $\kkdefFoam$ the generating object indexed by $0\leq k\leq N$ admits a decomposition as direct sum 
    \begin{equation*}
     k \simeq \bigboxplus_{\sum_i k_i=k} \bigboxtimes_{i} \kartwobj{k_i}{\lambda_i^{N_i}}
 \end{equation*}
 of monoidal products of objects $\kartwobj{k_i}{\lambda_i}$. 
\end{theoremC}
Moreover, the full sub-2-category on the objects $\kartwobj{m}{\lambda_i}$ is equivalent to $\kkar \kar N_i\mathbf{Foam}_+^{\Sigma_i}$ where $\Sigma_i$ contains $\lambda_i$ with multiplicity $N_i$ and no other root. Finally, the full sub-2-categories corresponding to distinct roots $\lambda_i$ are non-interacting in $\kkdefFoam$. This appears as \Cref{cor:cat_split} in the main text.

\subsection*{Acknowledgements}
The author would like to thank Christophe Hohlweg, David Reutter, Rapha\"el Rouquier, Catharina Stroppel and Kevin Walker for inspiring and helpful conversations on various aspects of this work. The author especially thanks her advisor Paul Wedrich for his insight, dedication and support.

\subsection*{Funding}

The author acknowledges support from the Deutsche
Forschungsgemeinschaft (DFG, German Research Foundation) under Germany's
Excellence Strategy - EXC 2121 ``Quantum Universe'' - 390833306 and the
Collaborative Research Center - SFB 1624 ``Higher structures, moduli spaces and
integrability'' - 506632645.

\section{Categorical completions}
Here we recall notions related to idempotent completions of 1- and 2-categories, specializing the constructions in \cite{GJF19} for $n = 2$, and comparing with \cite{DR19}. Our treatment follows \cite{D22}, where the $n = 2$ case is developed explicitly. Since the terminology is not uniform in the literature\footnote{We apologise in advance for introducing even more variance, but hope that our terminological choices are maximally descriptive.}, we opt for a self-contained treatment. 

\begin{definition}
    Let $\CS$ be a 1-category. An idempotent morphism $e \colon c \to c$ in $\CS$ is said to \emph{split} if there exist an object $d \in \CS$ and morphisms $i \colon d \to c$ and $r \colon c \to d$ such that $i\circ r = e$ and $r \circ i = \id_{d}$, i.e. $d$ is a \emph{retract} of $c$. A category is called \emph{idempotent-complete} if all idempotents split. 
\end{definition}

\begin{definition}
An \emph{idempotent completion} of a 1-category $\CS$ is a fully faithful functor $\iota_\CS \colon \CS \to \kar \CS$ such that $\kar \CS$ is idempotent-complete and every object $d \in \kar \CS$ is a retract of an object $\iota_\CS(c)$ for $c\in \CS$.
\end{definition}
Equivalently, an idempotent completion of $\CS$ is a functor that is initial among all functors from $\CS$ to idempotent-complete categories. This universal property ensures that idempotent completions of $\CS$ are unique up to equivalence. Existence can be proven via the following concrete construction.

\begin{remark} An idempotent completion $\kar \CS$ of $\CS$ is given by the 1-category consisting of: 
    \begin{itemize}
        \item objects given by pairs $(c, e)$ with $c\in \CS$ and $e\in \End_{\CS}(c)$ idempotent;
        \item morphisms $\text{Hom}_{\kar \CS}((c, e), (d, g)) = \{ f \in \text{Hom}_{\CS}(c, d) \mid gfe = f \}$ and
        \item composition inherited from $\CS$.
    \end{itemize}
    The fully faithful functor $\iota_\CS \colon \CS \to \kar \CS$ is given by sending any object $c \in \CS$ to $(c, \id_c)$ and a morphism $f \colon c \to d $ to itself. To verify that every idempotent splits, let $q \colon (c,e) \to (c,e)$ be an idempotent morphism. Take $i = q$ and $r = q$. Then, $i\circ r = q$ and $r\circ i = q = \id_{(c,e)}$ since $q$ is idempotent. 
\end{remark}
 We will also call the category $\kar \CS$ \emph{the} idempotent completion of $\CS$. Taking the idempotent completion is functorial and itself an idempotent operation. Let $\Cat$ denote the category of all small categories and functors between them.

\begin{proposition}[\cite{H85}]
    Idempotent completion is the idempotent monad
    \[ \kar \colon \Cat \to \Cat\]
    corresponding to the adjunction between the forgetful functor from categories to semicategories and its right adjoint.
\end{proposition}

\begin{conv}
    By 2-categories we mean bicategories (weak 2-categories) in the sense of \cite{B67}, without strictness assumptions, although we will often suppress associators and unitors. In diagrams, we read 1-morphisms from right to left and 2-morphisms from bottom to top. 
\end{conv}

\begin{definition}
    \label{def:2idempotent}
    Let $\CCS$ be a 2-category. A \emph{2-categorical idempotent} (\emph{2-condensation monad} in \cite[Definition 2.2.1]{GJF19}) in $\CCS$ consists of
    \begin{itemize}
        \item an object $c \in \CCS$;
        \item a 1-morphism $e \colon c \ot c$ and
        \item two 2-morphisms $\mu \colon e \circ e \Rightarrow e$ , $\delta \colon e \Rightarrow e \circ e$
    \end{itemize}
    such that the pair $(e, \mu)$ is a (non-unital) associative algebra, the pair $(e, \delta)$ is a (non-counital) coassociative coalgebra, $\delta$ is an $e$-bimodule map, $\mu \cdot \delta = \id_{e}$, and the \emph{Frobenius relation} expressed by the following commutative diagram holds 

    \[\begin{tikzcd}
	{e \circ e\circ e} && {e \circ e} && {e\circ e \circ e} \\
	&& e \\
	&& {e\circ e}
	\arrow["{\id_e \circ \mu}"',double, from=1-1, to=3-3]
	\arrow["{\delta \circ \id_e}"',double, from=1-3, to=1-1]
	\arrow["{\id_e \circ \delta}",double, from=1-3, to=1-5]
	\arrow["\delta",double, from=1-3, to=2-3]
	\arrow["{\mu \circ \id_e}",double, from=1-5, to=3-3]
	\arrow["\mu",double, from=2-3, to=3-3]
    \end{tikzcd}\]
    We denote the data of a 2-categorical idempotent as a tuple $(c, e, \mu, \delta)$ or, when $e, \mu, \delta$ are clear from context or irrelevant, simply by $c$. 
\end{definition}

The relation $\mu \cdot \delta = \id_{e}$ expresses that $\delta$ defines a section of the multiplication $\mu$ and the commutative diagram says that $\delta$ is a bimodule map. In other words, $(c,\mu)$ is a (non-unital) \emph{separable} algebra.

\begin{definition}
\label{def:manifestlysplit}
    Let $\CCS$ be a 2-category. A \emph{manifestly split 2-categorical idempotent} (\emph{2-condensation} in {\cite[Definition 2.1.1]{GJF19}}) in $\CCS$ consists of
    \begin{itemize}
        \item two objects $c, d \in \CCS$;
        \item two 1-morphisms $f \colon d \ot c$ and $g \colon c \ot d$;
        \item two 2-morphisms $\phi \colon f \circ g \Rightarrow \id_d$ and $\gamma \colon \id_d \Rightarrow f \circ g $ such that $\phi \cdot \gamma = \id_{\id_{d}}$.
    \end{itemize}
\end{definition}

The terminology is justified since for a manifestly split 2-categorical idempotent $(c,d,f,g,\phi,\gamma)$, the data $(c,g\circ f, \mu, \delta)$ with $\mu := \id_g \circ \phi \circ \id_f$ and $\delta:=\id_g \circ \gamma \circ \id_f$ defines a 2-categorical idempotent and we see the object $d$ as a \emph{2-categorical retract} of $c$. 

\begin{definition} 
\label{def:2idcomp}
Let $\CCS$ be a 2-category.
\begin{enumerate}
    \item A \emph{splitting} of a 2-categorical idempotent $(c, e, \mu, \delta)$ in a 2-category $\CCS$ is a manifestly split 2-categorical idempotent $(c,d,f,g,\phi, \gamma)$ together with a 2-isomorphism $\theta \colon g \circ f \cong e$ such that 
    \begin{align*}
        \mu = \theta \cdot (\id_g \circ \phi \circ \id_f) \cdot (\theta^{-1} \circ \theta^{-1})\\
        \delta = (\theta \circ \theta) \cdot (\id_g \circ \gamma \circ \id_f) \cdot \theta^{-1}.
    \end{align*}
    \item A 2-categorical idempotent in $\CCS$ is called \emph{split} if it admits a splitting in $\CCS$. 
\item The 2-category $\CCS$ is called \emph{locally idempotent-complete} if its Hom-categories are idempotent-complete. 
    \item The 2-category $\CCS$ is called \emph{idempotent-complete} (\emph{has all condensates} \cite{GJF19}) if it is locally idempotent-complete and every 2-categorical idempotent is split.
    \end{enumerate}
\end{definition}
For any 2-category $\CCS$, one can replace the $\Hom$-categories by their idempotent completions to obtain a locally idempotent-complete 2-category that we denote by $\kar \CCS$.

The idempotent completion of a locally idempotent-complete 2-category $\CCS$ is designed to be the smallest 2-category containing $\CCS$, in which every 2-categorical idempotent splits. To describe it explicitly, we need a few auxiliary notions.

\begin{definition}
\label{def:bimod}
    Let $c_1\equiv (c_1, e_1, \mu_1, \delta_1)$ and $c_2\equiv (c_2, e_2, \mu_2, \delta_2)$ be two 2-categorical idempotents in a 2-category $\CCS$. A $(c_2, c_1)$-bimodule 
    %of 2-categorical idempotents 
     consists of the data of a 1-morphism $ f \colon c_2 \ot c_1$ and 2-morphisms 
    \begin{align*}
        \nu_f^r \colon f \circ e_1 \Rightarrow f & \qquad \qquad
        \nu_f^l \colon e_2 \circ f \Rightarrow f \\
        \beta_f^r \colon f \Rightarrow f \circ e_1 & \qquad \qquad
        \beta_f^l \colon f \Rightarrow e_2 \circ f 
    \end{align*}
    such that $(f, \nu^r, \nu^l)$ is an associative $(e_2, e_1)$-bimodule, $(f, \beta^r, \beta^l)$ is a coassociative $(e_2, e_1)$-bicomodule, $\nu^r$ and $\beta^l$ commute, $\nu^l$ and $\beta^r$ commute, $\beta^r \cdot \nu^r = \id_f$ and  $\beta^l \cdot \nu^l = \id_f$ and the following diagrams commute 

    \[\begin{tikzcd}[column sep = 5mm, row sep = 5mm]
	{f \circ e_1 \circ e_1} && {f \circ e_1} && {f \circ e_1 \circ e_1} \\
	&& f \\
	&& {f \circ e_1}
	\arrow["{\id_f \circ \mu_1}"',double, from=1-1, to=3-3]
	\arrow["{\beta_f^r \circ \id_{e_{1}}}"',double, from=1-3, to=1-1]
	\arrow["{\id_f \circ\delta_1}",double, from=1-3, to=1-5]
	\arrow["{\nu_f^r}",double, from=1-3, to=2-3]
	\arrow["{\nu_f^r \circ \id_{e_{1}}}",double, from=1-5, to=3-3]
	\arrow["{\beta_f^r}",double, from=2-3, to=3-3]
\end{tikzcd}, \begin{tikzcd}[column sep = 5mm, row sep = 5mm]
	{e_2 \circ e_2 \circ f} && {e_2 \circ f} && {e_2 \circ e_2 \circ f} \\
	&& f \\
	&& { e_2 \circ f}
	\arrow["{\mu_2 \circ \id_f}"', double, from=1-1, to=3-3]
	\arrow["{\id_{e_{2}} \circ \beta_f^l}"',double, from=1-3, to=1-1]
	\arrow["{\delta_2 \circ \id_f}",double, from=1-3, to=1-5]
	\arrow["{\nu_f^l}",double, from=1-3, to=2-3]
	\arrow["{\id_{e_{2}} \circ \nu_f^l}",double, from=1-5, to=3-3]
	\arrow["{\beta_f^l}",double, from=2-3, to=3-3]
\end{tikzcd}.\]

A \emph{bimodule morphism} between $(c_2, c_1)$-bimodules $f,g\colon c_2\ot c_1$ is a $2$-morphism $\alpha\colon f\To g$ in $\CCS$ that intertwines with the $c_2$- and $c_1$-(co-)actions in the obvious ways.

\end{definition}

\begin{construction} 
\label{cons:rel_tensor_prod}
Let $\CCS$ be a locally idempotent-complete 2-category, $c_1,c_2,c_3$ 2-categorical idempotents in $\CCS$, $f \colon c_2 \ot c_1$ a $(c_2, c_1)$-bimodule and $g \colon c_3 \ot c_2$ a $(c_3,c_2)$-bimodule. To obtain a $(c_3, c_1)$-bimodule $g \circ_{e_2} f \colon c_3 \ot c_1$ consider the 2-morphism $\varphi$ given as the following composite:

\[ g\circ f \xRightarrow{\id_g \circ \beta_f^l } g \circ e_2 \circ f \xRightarrow{\nu_g^r \circ \id_{f}} g \circ f   \]

Since $(c_2, e_2, \mu_2, \delta_2)$ is a 2-categorical idempotent and the actions are associative and coassociative, it follows that $\varphi\in \End_{\CCS}(g\circ f)$ is idempotent. Any splitting $h \colon c_3 \ot c_1$ of the idempotent $\phi$, the existence of which is guaranteed by local idempotent-completeness, inherits the structure of a $(c_3, c_1)$-bimodule, and thus defines the composite bimodule $g \circ_{e_2} f$.

\end{construction}

\begin{definition}
    \label{def:2id}
    Let $\CCS$ be a locally idempotent-complete 2-category. The 2-category with 
    \begin{itemize}
        \item objects given by 2-categorical idempotents in $\CCS$;
        \item 1-morphisms given by bimodules as in \Cref{def:bimod};
        \item 2-morphisms given by bimodule morphisms, and
        \item horizontal composition given by relative tensor product as in \Cref{cons:rel_tensor_prod}
    \end{itemize}
    is called the \emph{2-categorical idempotent completion} of $\CCS$ and denoted by $\kkar\CCS$.
\end{definition}

For every locally idempotent-complete 2-category $\CCS$, there is a canonical 2-functor
\[\iota_{\CCS} \colon \CCS \to \kkar \CCS\]
sending each object $c\in \CCS$ to the 2-categorical idempotent $(c, \id_c,*,*)$ and each 1-morphism $f\colon c_2 \ot c_1$ to the bimodule $(f,*,*,*,*)$, where $*$ are unitors that we suppress. Furthermore, the 2-functor $\iota_{\CCS}$ acts as the identity on 2-morphisms, since they are automatically bimodule morphisms.

This $2$-functor is initial among all 2-functors to idempotent-complete 2-categories. This can be turned into an essentially unique characterization of $\kkar \CCS$ by means of a universal property, see e.g. \cite[Definition 1.2.1 and Appendix A.2]{D22}.  We summarize a few important consequences which follow from \cite[Theorem 2.3.11]{GJF19}.

\begin{proposition} 
\label{prop:detectcomplete}
    Let $\CCS$ be a locally idempotent-complete $2$-category. Then
    \begin{enumerate}
        \item $\kkar \CCS$ is idempotent-complete, and
        \item $\iota_{\CCS} \colon \CCS \to \kkar \CCS$ is an equivalence of 2-categories if and only if $\CCS$ is idempotent-complete.
    \end{enumerate}
\end{proposition}

\begin{remark}
\label{rmk:reutterdouglas}
An alternative construction for the idempotent completion of a 2-category $\CCS$, as defined in \cite{DR19}, has 
    \begin{itemize}
        \item objects given by separable monads in $\CCS$;
        \item 1-morphisms given by monad bimodules between them;
        \item 2-morphisms given by bimodule maps between the 1-morphisms.
    \end{itemize}

Here, a monad in $\CCS$ consists of an object $c \in \CCS$, a 1-morphism $e \colon c \ot c$ admitting a 2-morphism $\mu' \colon e \circ e \To e$ and, in contrast to a 2-categorical idempotent, a 2-morphism $\eta \colon \id_c \To e$ such that $(e, \mu', \eta)$ is a \emph{unital} associative algebra. A monad is separable if there exists an $e$-bimodule map $\delta' \colon e \To e \circ e$ such that $\mu' \cdot \delta' = \id_{\id_{e}}$. As in the case of 2-categorical idempotents, there are notions of bimodules between monads, bimodule morphisms between those, and a way of composing bimodules at a separable monad.

By \cite[Theorem 3.3.3]{GJF19}, the above construction is equivalent to the 2-category from Definition \ref{def:2id}. 
\end{remark}

\subsection{Direct sums in 2-idempotent-complete categories}
Here we consider locally additive 2-categories, i.e. 2-categories whose Hom categories are additive. Such 2-categories have a meaningful notion of direct sums of objects, which naturally arises upon considering 2-categorical idempotent completions.

\begin{definition}
Let $\CCS$ be a locally additive 2-category and $x_i\in \CCS$ for $i\in I$ a finite collection of objects. A \emph{direct sum} of the $x_i$ is an object $\boxplus_{i \in I} x_i\in \CCS$ together with inclusion 1-morphisms $\iota_j\colon \boxplus_{i \in I} x_i \ot x_j $ and projection 1-morphisms $\rho_j\colon x_j \ot \boxplus_{i \in I} x_i$ for every $j\in I$ such that:
\begin{itemize}
    \item $\rho_j\circ \iota_j \cong \id_{x_j}$ for every $j\in I$.
    \item $\rho_j\circ \iota_k \cong 0$ for $j\neq k\in I$. 
    \item $\id_{\boxplus_{i \in I} x_i} \cong \bigoplus_{i\in I} \iota_i \circ \rho_i$.
\end{itemize}
\end{definition}

\begin{proposition}[{\cite[Proposition 1.3.16]{DR19}}]
\label{prop:idem_sum}
Let $\CCS$ be a 2-idempotent-complete, locally additive 2-category and $x \in \CCS$ an object. If $\id_x \cong \bigoplus_{i \in I} f_i$ is a finite decomposition of the identity 1-morphism on $x$ into nonzero 1-morphisms, then there is a finite decomposition $x \cong \boxplus_{i \in I} x_i$ of $x$ into nonzero objects $x_i\in \CCS$ with inclusion 1-morphisms $\iota_j\colon \boxplus_{i \in I} x_i \ot x_j $ and projection 1-morphisms $\rho_j\colon x_j \ot \boxplus_{i \in I} x_i$, such that $f_j \cong \iota_j\circ \rho_j$ and $\rho_j \circ \iota_j \cong \id_{x_i}$.
\end{proposition}

The following consequence will be important in \Cref{sec:foams}.

\begin{proposition}
\label{prop:idem_sumtwo}
   Let $\CCS$ be a 2-idempotent-complete, locally additive 2-category and $x \in \CCS$ an object. If the ring of 2-endomorphisms of $x$ admits a finite decomposition of the form $\End(\id_x) = \bigoplus_{i\in I} A_i$, then the object $x$ admits a direct sum decomposition $x\simeq \bigboxplus_{i \in I} x_i$ with $\End(\id_{x_i})\cong A_i$. 
\end{proposition}

\begin{proof}
Let $p_i\in \End(\id_x)$ denote the idempotent projecting onto $A_i$ and $f_i\colon x\ot x$ its image, which exists because $\CCS$ is locally idempotent-complete. Then $\id_{x}$ admits a direct sum decomposition $\id_x\cong \bigoplus_{i\in I} f_i$ with projections and inclusions given by the $p_i$, so the result follows from Proposition \ref{prop:idem_sum}.
\end{proof}

\section{Soergel bimodules and singular Soergel bimodules}
\label{sec:Sbim}

In this section $(W,S)$ will denote a Coxeter system with finite set of generators $S$ and $k$ a commutative ring. We recall the essential features of Soergel bimodules and singular Soergel bimodules and refer to \cite{MR4220642} for a textbook account. The main new result in this section is \Cref{prop:Sbimmonbicat} on semistrict monoidal structures on singular Bott-Samelson bimodules and singular Soergel bimodules in type $A$.

\subsection{Realizations}
\label{sec:real}
 A \emph{realization} $\mathfrak{h} := (V, \{ \alpha_s \colon s \in S \}, \{ \alpha^{\vee}_s \colon s \in S \})$ over $k$ consists of 
\begin{itemize}
    \item a free $k$-module $V$ of finite rank,
    \item a subset $\{ \alpha_s \colon s \in S \} \subset V$,
    \item a subset $\{ \alpha^{\vee}_s \colon s \in S \} \subset \Hom_k(V, k) = V^*$, 
\end{itemize}

with $\langle \alpha_s^\vee , \alpha_s \rangle = 2$ for all $s \in S$, such that the assignment 
    \[ s \mapsto  (v \mapsto v - \langle \alpha_s^\vee , v \rangle \alpha_s) \] defines a representation of $W$ on $V$ and an additional technical condition \cite[Eq. 3.3]{EW16} on the order $m_{st}$ for $s, t \in S, s \neq t$ is satisfied.

    A realization is \emph{reflection faithful} if the action of $W$ on $V$ is faithful and there is a bijection between the set of reflections and the codimension one subspaces of $V$ that are fixed by some element of $W$. 

     A realization is called \emph{balanced} if a certain extra condition on $m_{st}$ is satisfied, see \cite[Definition 3.7]{EW16}. For some comments on what happens to (singular) Soergel bimodules without these assumptions, we refer to \cite[Section 3.6]{EKLP24_demazure}

\begin{definition}
\label{def:r}
    For a reflection faithful realization $\mathfrak{h}$ of $(W,S)$, let
    \[R := \bigoplus_{m \geq 0} S^m({V})\] denote the graded symmetric $k$-algebra on $V$, such that $V$ is in degree 2, which inherits an action by $W$. For a subset $I\subset S$, let $W_I$ denote the parabolic subgroup of $W$ generated by $I$ and by 
    \[ R^{I}:=R^{W_I} := \{ f \in R \mid \forall w\in W_I: w\cdot f = f\} \]
    the subalgebra of $W_I$-invariants in $R$. For $s\in S$, we abbreviate $R^s:=R^{\{s\}}$.
\end{definition}

\begin{definition}
    \label{def:demazure_ops}
    %Take $(W,S)$, $s \in S$ and $R$ as in \Cref{def:r}. 
    For $s \in S$ we consider the \emph{Demazure operator} 
    \[ \partial_s \colon R \to R^s\langle -2 \rangle \;,\qquad \partial_s(f) := \frac{f - s\cdot f}{\alpha_s}. \]
    More generally, for a \emph{finitary} subset $I \subset S$, i.e. one that generates a finite parabolic subgroup $W_I\subset W$, one obtains a generalized Demazure operator as follows. Let $s_1 \cdots s_d$ be a reduced expression for the longest element $w_I$ of $W_I$. Then the map 
    \[ \partial_I \colon R \to  R^I \langle-2l(w_I) \rangle\;,\qquad \partial_I := \partial_{s_1} \partial_{s_2} \cdots \partial_{s_d}.\]
    is well-defined and independent of the choice of reduced expression, see \cite{D73,EKLP24_demazure}. 
\end{definition}

\begin{definition}
    A realization satisfies \emph{Demazure surjectivity} if $\partial_s$ is surjective onto $R^s$ for all $s \in S$. It satisfies \emph{generalized Demazure surjectivity} if the map $\partial_I$ is surjective onto $R^I$ for all finitary $I \subset S$. 
\end{definition}

\subsection{Frobenius extensions}
\label{sec:Frob}

We recall the notion of a graded Frobenius extension, which underlies some of the key features of Soergel bimodules and singular Soergel bimodules. 
\begin{definition}
    A graded Frobenius extension of degree $l$ is an inclusion of graded commutative rings $X \subset Y$ together with a non-degenerate $X$-linear map 
    \[\delta_X^{Y} \colon Y \to X\]
    called the Frobenius trace, homogenous of degree $-2l$. 
    Non-degeneracy requires $Y$ to be free of finite rank as an $X$-module and that there exist homogeneous finite bases $\{c_i\}_i$ and $\{d_j\}_j$ of $Y$ over $X$ such that 
    \[\delta_X^{Y}(c_i d_j) = \delta_{ij}.\]
\end{definition}

\begin{example}
    \label{ex:bs_frob_ext}

    Consider the Coxeter system of type $A_{n-1}$, which yields a presentation of the symmetric group $S_n$ with generators given by the simple transpositions $(i, i+1)$. We frequently refer to this setting in the text as the \emph{type $A$ setting}. The permutation representation of $S_n$ can be equipped with the structure of a reflection faithful balanced realization. For $k$ a field of characteristic zero, (generalized) Demazure surjectivity is satisfied \cite[Lemma 3.6]{EKLP24_demazure}.  Let $s=(i,i+1) \in S$ denote a simple transposition. Then $R^s \subset R = k[x_1, \ldots, x_n]$ is a graded Frobenius extension of degree $1$ with Frobenius trace given by the Demazure operator, acting on a polynomial $p \in R$ by 
    \[ \partial_s(p) = \frac{p - s(p)}{x_i - x_{i+1}}, \]
    where $s(p)$ denotes the action of $s$ on $p$. A choice of dual bases for this extension is $(1, x_i)$ and $(-x_{i+1}, 1)$

For $J \subset S$ the inclusion $R^J \subset R $ is also a graded Frobenius extension of degree $l(w_J)$ thanks to generalized Demazure surjectivity \cite[Theorem 4.2]{EKLP24_demazure}. For any reduced expression of the longest element $w_J = s_1 \cdots s_d$ in terms of $s_i\in S$, the Frobenius trace of choice is precisely the generalized Demazure operator from \Cref{def:demazure_ops}. 
\end{example}

\subsection{Adjunctions from Frobenius extensions}

The following four bimodule maps can be associated to any (graded) Frobenius extension $X\subset Y$ with trace $\delta^Y_X$ and dual bases $\{c_i\}_i$ and $\{d_j\}_j$:
\begin{equation}
\label{eq:frob_ext_maps}
\begin{aligned}
    i& \colon {}_XX_X 
    \to {}_XY_X  
    \qquad &&1\mapsto 1  \\
    %%%
   \delta& \colon {}_XY_X 
   \to {}_XX_X 
    \qquad &&p \mapsto \delta_X^{Y}(p)\\
    %%%
        m& \colon {}_YY \otimes_{X} Y_Y 
        \to {}_YY_Y 
    \qquad &&p \otimes q \mapsto p \cdot q\\
    %%%
    d& \colon {}_YY_Y 
    \to {}_YY \otimes_X Y_Y
    \qquad &&1 \mapsto \textstyle \sum_i c_i \otimes d_i
\end{aligned} 
\end{equation}

\begin{lemma}
\label{lem:indresadj}
    Let $X \subset Y$ be a graded Frobenius extension of degree $l$ with trace $\delta^Y_X$. Consider the functors 
\begin{equation}
\begin{aligned}
 \Ind_X^Y &\colon X\gmod \to Y\gmod \qquad && M \mapsto {}_YY \otimes_{X} M\\
  \Res_X^Y &\colon Y\gmod \to X\gmod \qquad && M \mapsto {}_XY  \otimes_Y M \langle l \rangle,
\end{aligned} 
\end{equation}
where $\langle l \rangle$ denotes a grading shift by $l\in \Z$. Then  $\Ind_X^Y$ is left adjoint to $\Res_X^Y\langle -l \rangle$ and right adjoint to  $\Res_X^Y\langle l \rangle$. 
\[
\cdots \dashv \Res_X^Y \langle l \rangle\;\;
\dashv \;\; \Ind_X^Y\;\; \dashv\;\;
\Res_X^Y \langle -l \rangle\;\;\dashv\;\;
 \cdots
\]
\end{lemma}
\begin{proof}
    The maps $i$ and $m$ exhibit $\Ind_X^Y$ as left adjoint to $\Res_X^Y$, while $\delta$ and $d$ exhibit $\Res_X^Y$ as left adjoint to $\Ind_X^Y$. 
\end{proof}

$\Ind_{X}^Y$ and $\Res_X^{Y}$ will often be considered as the bimodules ${}_YY_X$ and ${}_XY_Y \langle l \rangle$, respectively. 

\subsection{Soergel bimodules}
We retain notation and conventions from Section~\ref{sec:Frob}. We fix a Coxeter system $(W,S)$ and assume that we are working with a balanced, reflection faithful realization satisfying generalized Demazure surjectivity. 

\begin{definition}    
    For a simple transposition $s\in S$ consider $B_{s} = R \otimes_{R^s} R \langle 1 \rangle$. Any bimodule of the form
    \[B_{\underline{w}}\langle l \rangle := B_{s_1} \otimes_R B_{s_2} \otimes_R \cdots \otimes_R B_{s_k} \langle l \rangle,\]
    where $\underline{w} = (s_1, s_2 \ldots, s_k)$ is a word in $S$ and $l\in \Z$, is called a \emph{Bott--Samelson bimodule}. Here we allow the empty word $w=\emptyset$ and declare $B_\emptyset:=R$.

    The monoidal \emph{category of Bott--Samelson bimodules} $\BSbim$ is the full subcategory of graded $R$-bimodules with objects given by the Bott--Samelson bimodules. 
\end{definition}

\begin{definition}
    \label{def:sbim} The monoidal \emph{category of Soergel bimodules} is the (graded) additive idempotent completion $\SBim:= \add{(\kar \BSbim)}$ of the category of Bott--Samelson bimodules. Equivalently, it is the full subcategory of graded $R$-bimodules on directs sums of direct summands of shifts of Bott--Samelson bimodules. Its objects are called \emph{Soergel bimodules}.
\end{definition}

\begin{theorem}[Soergel's Categorification Theorem]
    \label{thm:soergel_cat}
    The indecomposable objects of $\SBim$ are classified up to isomorphism by $W \times \Z$. There is an isomorphism of $\Z[q,q^{-1}]$-algebras $H \cong K_0(\SBim)$, where $H$ is the Hecke algebra associated to $W$ and $K_0$ is the split Grothendieck ring. Under this isomorphism, the Kazhdan-Lusztig basis element $C_w$ corresponds to the class of the indecomposable object $B_{w}$, which appears as retract of the Bott--Samelson bimodule $B_{\underline{w}}$ for any reduced expression $\underline{w}$ of $w$ but not as a retract of a Bott--Samelson bimodule for any shorter expression.
\end{theorem}

\begin{remark}
    The assumption of a reflection faithful realization was imposed by Soergel in his original definition and the theorem above is proven in this setting. The assumption of a balanced representation does not interact with this result. 
\end{remark}

The monoidal category $\SBim$ can be interpreted as a 2-category with a single object $R$, 1-morphisms given by Soergel bimodules and horizontal composition given by the tensor product over $R$. We denote this 2-category by $\sbim$.

Soergel bimodules admit the following generalization first defined in \cite{W11}. 

\begin{definition}
For a sequence 
\[
I_1 \subset J_1 \supset I_2 \subset J_2 \supset \cdots \supset I_k \subset J_{k} \supset I_{k+1} 
\]
of subsets of $S$ for some $k\in \N_0$ and $l \in \Z$, the $(R^{I_1},R^{I_{k+1}})$-bimodule
\[ R^{I_1} \otimes_{R^{J_1}} R^{I_2} \otimes_{R^{J_2}} \cdots \otimes_{R^{J_k}} R^{I_{k+1}} \langle l \rangle\]
is called a \emph{singular Bott--Samelson bimodule}. 
\end{definition}

It is easy to see that singular Bott--Samuelson bimodules are isomorphic to bimodules constructed as grading shifts of iterated relative tensor products of various induction and restriction bimodules for the inclusions $R^J\subset R^I$ with $I\subset J\subset S$. 

\begin{definition}
    The \emph{2-category of singular Bott--Samelson bimodules} $\sBSbim$  is the $2$-category with: 
    \begin{itemize}
        \item objects given by the graded rings $R^J$ for $J \subset S$; 

        \item the category of 1-morphisms $R^{J} \ot R^{I}$ given by the full subcategory of $R^J\gbimod R^I$ on the singular Bott--Samelson bimodules;
        \item horizontal composition is given by tensor product over intermediate rings. 

    \end{itemize}

\end{definition}

\begin{definition}
    \label{def:ssbim}
   The \emph{2-category of singular Soergel bimodules}, denoted by $\sSbim$, has the same objects as $\sBSbim$ and the Hom categories are obtained as (graded) additive idempotent completions of the Hom categories between singular Bott--Samelson bimodules:
   \[
\Hom_{\sSbim}({R^I, R^J}):= \add{(\kar \Hom_{\sBSbim}({R^I, R^J}))}. 
   \]
   Equivalently, this Hom category can be described as the full subcategory of $R^J\gbimod R^I$ containing all direct sums of direct summands of shifts of 1-morphisms in $\sBSbim$. 
\end{definition}

\begin{theorem}[Soergel-Williamson Categorification Theorem]
    \label{thm:soergel_williamson}
    For $I, J \subset S$, the indecomposable objects of $\Hom_{\sSbim}({R^I, R^J})$ are classified up to isomorphism and grading shift by double cosets $W_J \backslash W / W_I$, with distinguished representatives $\{\prescript{I}{}{}{B_p}^J \mid p \in W_J \backslash W / W_I \}$ for each isomorphism class that decategorify to a basis of the Hecke algebroid.
\end{theorem}
In the type $A$ setting, Soergel bimodules form a locally $k$-linear monoidal bicategory, as we now recall. For $n\in \N_0$ we let $\BSBim_{n}$ $\SBim_{n}$, $\sBSbim_{n}$ and $\sSbim_{n}$ denote the monoidal categories of Bott-Samelson bimodules and Soergel bimodules as well as the 2-categories of singular Bott-Samelson bimodules and and singular Soergel bimodules, all associated to the permutation representation of the symmetric group $S_n$ in characteristic zero.

\begin{proposition}[{\cite[Proposition 2.13]{SW24}}]
    \label{sbim_monoidal}
    The monoidal Soergel bimodule categories $\SBim_{n}$ for all $n\in \N_0$ assemble into a monoidal bicategory $\msbim$ with 
    \begin{itemize}
        \item objects given by natural numbers $\N_0$;
        \item the endomorphism monoidal category of $n$ is $\SBim_{n}$ and all other Homs are trivial;
        \item the monoidal structure given on objects by $m\boxtimes n:=m+n$ and on morphisms by parabolic induction.

    \end{itemize}

\end{proposition}

In the type $A$ setting, singular Soergel bimodules also form such a monoidal bicategory. We develop the construction sketched in \cite{HRW21} by adapting tools from \cite{SW24} to the singular setting. In the following constructions, when $p\in\N_0$ we let $R_p$ be the ring of polynomials in $p$ variables $k[x_1, \ldots,  x_p]$.

\begin{definition}
    Given $a, b, c \in  \N$, the \emph{variable shifting morphisms} are the algebra homomorphisms 
    \[
    j_{a|c} \colon R_b \to R_{a + b+ c}\;,\qquad x_i \mapsto x_{i+a}
    \]
    If $R_b^{J}$ is the ring of invariants under the parabolic subgroup generated by $J = \{s_{i_1}, \ldots , s_{i_l} \}$ of ${S_b}$, then $j_{a|c}$ restricts to an algebra homomorphism   
    \begin{align*}
        j^{J}_{a|c} \colon R_b^J \to R^{aJc}_{a+b+c}
    \end{align*}
    for the set of simple transpositions $aJc =\{s_{{(i+a)}_1}, \ldots , s_{(i+a)_l} \} \subset S_{a+b+c}$. We may simplify this notation in the following cases: $Jc:=0Jc$ and $aJ:=aJ0$ and $J:=0J0$.
     Likewise, if $R_a^A$ and $R_c^C$ are rings of invariants corresponding to parabolic subgroups $W^A \subset S_a$ and $W^C \subset S_c$, then the above algebra homomorphism factors through
     \[R^{AJC}_{a+b+c}:=R^{A(b+c) \cup aBc \cup (a+b)C}_{a+b+c}\]
     
\end{definition}

\begin{definition}
For an $(R_b^I,R_b^J)$-bimodule $M$ we consider the tensor product $R_a\otimes_k M \otimes_k R_c$, which is, via the obvious algebra isomorphisms, an $(R^{aIc}_{a+b+c},R^{aJc}_{a+b+c})$-bimodule. Likewise, we consider $R^A_a\otimes_k M \otimes_k R^C_c$, which is naturally an $(R^{AIC}_{a+b+c},R^{AJC}_{a+b+c})$-bimodule.
\end{definition}
These constructions define functors $j_{a|c}$ from $(R_b^I,R_b^J)$-bimodules to $(R^{aIc}_{a+b+c},R^{aJc}_{a+b+c})$-bimodules and functors $j_{A|C}$ from $(R_b^I,R_b^J)$-bimodules to $(R^{AIC}_{a+b+c},R^{AJC}_{a+b+c})$-bimodules.  All such functors are compatible with horizontal composition. Thus we get induced 2-functors:
\[
j_{a|c}, j_{A|C} \colon \sBSbim_b \to \sBSbim_{a+b+c}, \quad j_{a|c}, j_{A|C} \colon \sSbim_b \to \sSbim_{a+b+c}
\]
that act on objects by the index-shift $J\mapsto aJc$ resp. $J\mapsto A(b+c) \cup aJc \cup (a+b)C$.

\begin{definition}\label{monoidalstr}
Let $M$ be a $(R_m^K, R_m^I)$-bimodule and $N$ a $(R_n^L, R_n^J)$-bimodule. We define the $\boxtimes$-product of $M$ and $N$ as the $(R^{Kn\cup mL}_{m+n},R^{In\cup mJ}_{m+n})$-bimodule
\[M_m\boxtimes N_n := j_{0|L}(M_m) \otimes_{R^{In\cup mL}_{m+n}} j_{I|0}(N_n) \] 
and assign to a pair $f\colon M_m\to M'_m$, $g\colon N_n\to N'_n$ of bimodule morphisms the morphism
\[f \boxtimes g:= j_{0|L}(f) \otimes j_{I|0}(g) \colon M_m\boxtimes N_n \to M'_m\boxtimes N'_n\]
of bimodules. It is immediate that $\boxtimes$ is functorial in both arguments and sends pairs of singular Bott-Samelson bimodules to singular Bott-Samelson bimodules and likewise for singular Soergel bimodules. We call the resulting functors \emph{parabolic induction}:
\begin{equation}
    \label{eq:parabolicind}
    \boxtimes \colon \sBSbim_m \times \sBSbim_n \to \sBSbim_{m+n}
    \quad\text{and}\quad
    \boxtimes \colon \sSbim_m \times \sSbim_n \to \sSbim_{m+n} .
\end{equation}

\end{definition}

As in \cite[Lemma 2.12]{SW24} one can now provide explicit interchange isomorphism. As a consequence, one obtains the following singular analogue of \cite[Proposition 2.13]{SW24}.
    
\begin{proposition}
    \label{prop:Sbimmonbicat}
The parabolic induction functors \eqref{eq:parabolicind} endow the \emph{bicategory of singular Soergel bimodules}
\begin{equation}\label{SBimDef}
\mssbim:= \bigsqcup_{n\geq 0} \sSbim_n% \quad\text{and}\quad \DS:= \bigsqcup_{n\geq 0}\DS_n
\end{equation}
with the structure of a locally $k$-linear monoidal bicategory with
monoidal product denoted $\boxtimes$. By restriction, the same result holds for the \emph{bicategory of singular Bott--Samelson bimodules} $\msbsbim= \bigsqcup_{n\geq 0} \sBSbim_n $.
\end{proposition}

\section{Higher idempotents in Soergel bimodules}

In this section we return to considering (singular) Soergel bimodules for a fixed Coxeter system $(W,S)$ with respect to a balanced and reflection faithful realization satisfying generalized Demazure surjectivity.
Our goal now is to understand and relate the 2-categorical idempotent completions of the locally idempotent-complete 2-categories $\sbim$ and $\sSbim$. 

\subsection{Manifestly split 2-categorical idempotents from parabolics}
\label{sec:ssbimcompl}

For the following, let $J$ be a finitary subset and let $R^J$ denote its associated ring of invariants.

\begin{lemma}
\label{lem:manifestly}
  The following data defines a manifestly split 2-categorical idempotent (in the sense of \Cref{def:manifestlysplit}) in  $\sSbim$:
  \begin{itemize}
      \item the two objects $R$ and $R^J$;
      \item the restriction and induction bimodules ${}_{R^J}R_{R}$  and $g:={}_{R}R_{R^J}\langle l(w_J)\rangle$;
      \item the bimodule homomorphisms:
      \begin{equation}
    \label{eq:manifestly-split-maps}
    \begin{aligned}
       \phi \colon  {}_{R^J}R_{R^J} \to {}_{R^J}R^J_{R^J}, \qquad  p\mapsto \partial_J(p)
                \\
           \gamma\colon {}_{R^J}R^J_{R^J} \to {}_{R^J}R_{R^J}, \qquad 1 \mapsto \frac{\alpha_J}{|W_J|}
    \end{aligned}
\end{equation}
where $\alpha_J\in R$ denotes the polynomial obtained as the product of all positive roots of $W_J$. 
  \end{itemize}
\end{lemma}
\begin{proof}
The relation $\phi\cdot \gamma=\id_{\id_d}$ follows from $\partial_J(\alpha_J) = |W_J|$, see \cite[Theorem 3.4]{EKLP24_demazure}, \cite{D73}. 
\end{proof}

In particular, \Cref{lem:manifestly} implies that the singular Bott--Samelson bimodule 
\[B_J:={}_R R \otimes_{R^J} R_R\langle l(w_J)\rangle\]
determines a 2-categorical idempotent $(R, B_J, \mu, \delta)$ in $\sSbim$ with 2-morphisms: 
\begin{equation}
    \label{eq:2-idempotent-maps}
    \begin{aligned}
        \mu \colon  R \otimes_{R^J} R \otimes_{R} R \otimes_{R^J} R &\to  R \otimes_{R^J} R &   \delta \colon  R \otimes_{R^J} R &\to R \otimes_{R^J} R \otimes_{R^J} R &  \\
        f \otimes g \otimes h &\mapsto \partial_J(g)f \otimes h & \qquad    f \otimes g &\mapsto   f \otimes \frac{\alpha_J}{|W_J|} \otimes g
    \end{aligned}
\end{equation}

We thus see that the object $R^J$ is a 2-categorical retract of $R$ obtained by splitting the above 2-categorical idempotent on $B_J$. In fact, the latter already lives inside  Soergel bimodules.

\begin{lemma}
    \label{lemma:isomorphism_bwj}
    The singular Bott--Samelson bimodule $B_J$ is isomorphic to the indecomposable Soergel bimodule $B_{w_J}$ in the category of graded $R$-bimodules. 
\end{lemma}

\begin{proof} \cite[Theorem 3]{W11} states that for $I, J \subset S$, $p  \in W_{I} \backslash W / W_J$ and $p_{+}$ the unique longest element of $p$, there is an isomorphism
    \begin{align*}
        R \otimes_{R^I} \prescript{I}{}{}{B_p}^J \otimes_{R^J} R \cong B_{p_{+}}
    \end{align*}
    in the category of graded $R$-bimodules. We apply this isomorphism in the case when $p \in W_J \backslash W / W_J$ is the double coset $W_J e  W_J$ of the identity. Then $B_{p_{+}} = B_{w_J}$. The isomorphism $\prescript{J}{}{}{B_{e}}^J \cong R^J$ follows from \cite[Section 7.4]{W11}.
\end{proof}

We immediately obtain the following corollary.

\begin{corollary}
\label{cor:insbim}
    Every object $R^J$ of $\sSbim$ is a 2-categorical retract of a 2-categorical idempotent $(R, B_J, \mu, \delta)$ in $\sbim$.
\end{corollary}

\begin{theorem} \label{thm:main} The inclusion of Soergel bimodules into singular Soergel bimodules yields a commutative diagram
\[
\begin{tikzcd}
	       \sbim & \sSbim \\
	       \kkSBim & \kksSBim
	       \arrow["\iota", from=1-1, to=1-2]
	       \arrow["\iota_{\sbim}"', from=1-1, to=2-1]
	       \arrow["\iota_{\sSbim}", from=1-2, to=2-2]
	       \arrow["{\simeq}"', from=2-1, to=2-2]
    \end{tikzcd}
\]
    such that the bottom horizontal map is an equivalence of 2-categories. 
\end{theorem}
\begin{proof}
The 2-functor $\iota \colon \sbim \to \sSbim$ between locally idempotent-complete 2-categories is fully faithful in the sense that it induces equivalences on Hom-categories. Additionally, every object $R^J$ of $\sSbim$ is a 2-categorical retract of a 2-categorical idempotent on the object $R$ in the image of $\iota$ by Corollary~\ref{cor:insbim}. The result thus follows from  \cite[Lemma A.2.4]{D22}.
\end{proof}

\begin{remark}
From the description of the monoidal bicategories $\msbim$ and $\mssbim$, we expect there to be an equivalence of monoidal bicategories $\kkar \msbim \cong \kkar \mssbim$. As far as the structure of the bicategory is concerned, this is clear. We do not pursue compatibility with the monoidal structure here.
\end{remark}

\subsection{Idempotent-completeness and the Klein--Elias--Hogancamp conjecture}

Theorem~\ref{thm:main} has the following interpretation: 

\begin{corollary}
    \label{cor:main}
    The idempotent completion 2-functor $\sbim \to \kkSBim$ factors through $\sSbim$: singular Soergel bimodules are a \emph{partial} 2-categorical idempotent completion of $\SBim$.
\end{corollary} We are wondering \emph{how} partial.

\begin{proposition}
\label{prop:equiv}
    The following are equivalent:
    \begin{enumerate}
        \item The 2-category $\sSbim$ is idempotent-complete.
        \item The idempotent completion 2-functor $\sSbim \to \kksSBim$ an equivalence.
        \item The idempotent completion 2-functor $\sSbim \to \kksSBim$ is surjective on equivalence classes of objects.
        \item The composition $\sSbim \to \kksSBim \xrightarrow{\simeq} \kkSBim$ is an equivalence, 
        i.e. singular Soergel bimodules are the \emph{full} 2-categorical idempotent completion of Soergel bimodules.
        \item The composition $\sSbim \to \kksSBim \xrightarrow{\simeq} \kkSBim$ is surjective on equivalence classes of objects.
    \end{enumerate}
\end{proposition}
\begin{proof}
The equivalence (1)$\iff$(2) is a consequence of \Cref{prop:detectcomplete}. Since idempotent completion is always fully faithful, it is an equivalence (2,4) as soon as it is essentially surjective, i.e. surjective on equivalence classes of objects (3,5). The implication (2)$\implies$(4)  
follows by composing with the (reverse) equivalence from \Cref{thm:main}. Finally, the implication (4)$\implies$(1) is immediate: $\sSbim$ is idempotent complete if it is equivalent to the completion $\kkSBim$ of $\sbim$. 
\end{proof}

\begin{question}
\label{q:idempotentcompletion}
Is any of the equivalent conditions of \Cref{prop:equiv} satisfied for any non-trivial type $(W,S)$ and a suitable realization?
\end{question}

Even in type $A$ we are unsure if this question has a positive answer. As a sanity check, one can ask whether singular Soergel bimodules can account for some of the \emph{expected} 2-categorical idempotents in Soergel bimodules in type $A$. 
\smallskip

From several perspectives, it is expected that indecomposable Soergel bimodules $B_d$ corresponding to distinguished involutions $d\in \mathcal{D} \subset W$ can be equipped with the structure of a Frobenius algebra object in $\SBim$. This goes back to \cite[Section 5.2]{K14} and \cite[Conjecture 4.41]{EH17}, see also \cite[Section 4.4]{MMMTZ23} where it appears under the name \emph{Klein--Elias--Hogancamp conjecture}.

The set of distinguished elements is defined by
\begin{equation}
    \label{eqn:Duflo} 
    \mathcal{D}=\{w\in W\mid a(w)=\Delta(w)\}
    \end{equation}
    for Lusztig's functions $a,\Delta\colon W \to \N_0$ \cite[\S 14, \S 13, respectively]{L14}\footnote{We always refer to Lusztig's updated version of this reference: \href{https://arxiv.org/abs/math/0208154}{ 	arXiv:math/0208154}} and it is known that $w\in \mathcal{D}$ are necessarily involutions \cite[\S 14 (P6)]{L14}\footnote{The famous properties (P1)-(P15) are stated as conjectures in \cite[\S 14]{L14} and subsequently proven in type A. For Hecke algebras (with equal parameters) for other Coxeter groups they also hold as consequence of \cite{EW14}, see  \cite[Introduction to \S 4]{EH17}}. The values $-a(w)$ and $\Delta(w)$ can be interpreted as the minimal degrees of projections $B_w\otimes B_w \twoheadrightarrow B_w$ onto direct summands \cite[Corollary 4.12 and Conjecture 4.41]{EH17} resp. nontrivial maps $R\to B_w$ \cite[Proposition 4.33]{EH17}. Such maps have a chance to serve as multiplication and unit in a graded Frobenius algebra structure on $B_w$ only if they have complementary degrees, i.e. if $w\in \mathcal{D}$. One can strengthen the Elias--Hogancamp--Klein conjecture by asking for separability:
    
    \begin{question}
    \label{q:separable}
    Does $B_d$ for $d\in \mathcal{D}$ admit the structure of a \emph{separable} Frobenius algebra object in $\SBim$, i.e. an object of $\kkSBim$?    
    \end{question}

    If the answer to \Cref{q:separable} is positive for some $d\in \mathcal{D}$, one can further ask:
    
    \begin{question}
    \label{q:sepsplits}
       Is the $2$-categorical idempotent represented by a given separable Frobenius algebra structure on $B_d$ manifestly split in $\sSbim$?
    \end{question}

    This is again too much to hope for, but also not necessary for a positive answer to \Cref{q:idempotentcompletion}. 

   \begin{question}
    \label{q:sepsequiv}
       Is the $2$-categorical idempotent represented by a given separable Frobenius algebra structure on $B_d$ equivalent in $\kksSBim$ to an object of $\sSbim$?
    \end{question}

We now specialize the discussion to type $A$, i.e. symmetric groups $S_n$ with their permutation representations in characteristic zero. In this case, all involutions are distinguished since every left cell contains a unique involution, which must be distinguished by \cite[\S 14 (P13)]{L14}.

\smallskip 

For $n \leq 3$ all involutions in $S_n$ are longest elements of parabolic subgroups. The corresponding indecomposable Soergel bimodules indeed carry separable Frobenius algebra structures and we have seen in \Cref{sec:ssbimcompl} that the corresponding 2-categorical idempotents split in $\sSbim_n$. Thus \Cref{q:separable} and \Cref{q:sepsplits} have positive answers in this case. However, we do not currently know if these are all (relevant) examples:

\begin{question}
    Do all 2-categorical idempotents of $\sbim_n$ for $n\leq 3$ split in $\sSbim_n$?
\end{question}

For $n\geq 4$ the symmetric groups contain involutions that are not longest elements of parabolics.  We consider $S_4$ as generated by the simple transpositions $s = (1,2)$, $t = (2,3)$, and $u = (3,4)$. Then $S_4$ contains the involutions $1$, $s$, $t$, $u$, $sts$, $tut$, $su$ and $stusts$, which are longest elements of parabolics, but also $tsut$ (in the same 2-sided Kazhdan–Lusztig cell as $su$) or $ustsu$ (in the same 2-sided Kazhdan–Lusztig cell as $sts$), which are not longest elements of parabolics. In the following we show that \Cref{q:separable} still has a positive answer for $S_4$

\begin{example}
\label{ex:tsut}
   In $\SBim_4$ we construct a 2-categorical idempotent on the indecomposable Soergel bimodule $B_{tsut}$. The Kazhdan-Lusztig basis elements satisfy 
        $C_{tsut} = C_{t}C_{s}C_{u}C_{t}$.
    Permutations satisfying this property are called \emph{Deodhar elements} in \cite{BW01}. As a consequence of Theorem~\ref{thm:soergel_cat}, the Bott-Samuelson bimodule 
\[ B_{tsut} \cong  B_{t} \otimes B_{s} \otimes B_{u} \otimes B_{t} \cong  B_{t} \otimes B_{su} \otimes B_{t} \]
    is indecomposable and isomorphic to the Soergel bimodule $B_{tsut}$. We show that it inherits its separable Frobenius algebra structure from that of $B_{su}$, which exists since $su$ is a longest element of a parabolic.
    
    The multiplication map $\mu_{tsut}$ and its section $\delta_{tsut}$ for $B_{tsut}$ are best described diagrammatically. For this we use the Elias-Khovanov-Williamson diagrammatic calculus for morphisms between Bott--Samelson bimodules with colour coding
     $\{s, t, u\} = \{ \begin{tikzpicture}[anchorbase,smallnodes,scale=.3] 
       \fill[rd] (0.5,0.5) circle (2.5mm);
   \end{tikzpicture}$, $\begin{tikzpicture}[anchorbase,smallnodes,scale=.3] 
       \fill[bl] (0.5,0.5) circle (2.5mm);
   \end{tikzpicture}$, $\begin{tikzpicture}[anchorbase,smallnodes,scale=.3] 
       \fill[gr] (0.5,0.5) circle (2.5mm);
   \end{tikzpicture} \}$:

\begin{equation*}
    \label{diag:mutsut}
    \mu_{tsut} = 
    \begin{tikzpicture}[anchorbase,smallnodes,scale=.3] 
        \draw[bl, line cap = round]   (0,0) \pu (2,3)
                        (7,0) \pu (5,3)
                        (3,0) \ur (3.5,0.5) \rd (4,0);
        \draw[rd, line cap = round]   (1,0) \uur (3,2)
                        (5,0) \uul (3,2)
                        (3,2) \pu (3,3); 
        \draw[gr, line cap = round]   (2,0) \uur (4,2)
                        (6,0) \uul (4,2)
                        (4,2) \pu (4,3);    
    \end{tikzpicture}
    \qquad \qquad
    \delta_{tsut} = 
    \begin{tikzpicture}[anchorbase,smallnodes,scale=.3] 
        \draw[bl, line cap = round]  (2,-3) \pu (0,0) 
                        (5,-3) \pu (7,0)
                        (4,0) \dl (3.5,-0.5) \lu (3,0);
        \fill[gy, line cap = round] (3.5,0) circle (2mm); 
        \draw[rd, line cap = round]   (3,-2) \uru (5,0)
                        (3,-2) \ulu (1,0)
                        (3,-3) \pu (3,-2); 
        \draw[gr, line cap = round]   (4,-2) \uru (6,0) 
                        (4,-2) \ulu (2,0) 
                        (4,-3) \pu (4,-2);   
    \end{tikzpicture}
    \quad \text{with} \quad
    \begin{tikzpicture}[anchorbase,smallnodes,scale=.3] 
        \fill[gy] (3.5,0) circle (3mm); 
    \end{tikzpicture}
     = \frac{1}{4} 
    \begin{tikzpicture}[anchorbase,smallnodes,scale=.2] 
            \draw[bl] (0,-1.5) \lu (-1.5,0) \ur (0,1.5) \rd (1.5,0) \dl (0,-1.5); 
            
            \fill[rd] (-0.5,0.5) circle (2.5mm); 
            \fill[rd] (-0.5,-0.5) circle (2.5mm); 
            \draw[rd] (-0.5,0.5) \pu (-0.5,-0.5);
            
            \fill[gr] (0.5,0.5) circle (2.5mm); 
            \fill[gr] (0.5,-0.5) circle (2.5mm); 
            \draw[gr] (0.5,0.5) \pu (0.5,-0.5);
            
    \end{tikzpicture}
\end{equation*}

Note that $\delta_{tsut}$ is just the reflection of $\mu_{tsut}$, with a  polynomial correction given by
    \[ \begin{tikzpicture}[anchorbase,smallnodes,scale=.3] 
        \fill[gy] (3.5,0) circle (3mm); 
    \end{tikzpicture}
     = \frac{1}{4}\left( (x_1 - x_2)(x_3 -x_4) - (x_1 - x_3)(x_2 - x_4) \right) \]
     
     For comparison, the multiplication $\mu_{su}$ and section $\delta_{su}$ for $B_{su}$ are given by:
\begin{equation*}
    \label{diag:musu}
    \mu_{su} = 
    \begin{tikzpicture}[anchorbase,smallnodes,scale=.3] 
        \draw[rd, line cap = round]   (1,0) \uur (3,2)
                        (5,0) \uul (3,2)
                        (3,2) \pu (3,3); 
        \draw[gr, line cap = round]   (2,0) \uur (4,2)
                        (6,0) \uul (4,2)
                        (4,2) \pu (4,3);    
    \end{tikzpicture}
    \qquad \qquad
    \delta_{su} = 
    \begin{tikzpicture}[anchorbase,smallnodes,scale=.3] 
        \fill[gy, line cap = round] (3.5,0) circle (2mm); 
        \draw[rd, line cap = round]   (3,-2) \uru (5,0)
                        (3,-2) \ulu (1,0)
                        (3,-3) \pu (3,-2); 
        \draw[gr, line cap = round]   (4,-2) \uru (6,0) 
                        (4,-2) \ulu (2,0) 
                        (4,-3) \pu (4,-2);   
    \end{tikzpicture}
    \quad \text{with} \quad
    \begin{tikzpicture}[anchorbase,smallnodes,scale=.3] 
        \fill[gy] (3.5,0) circle (3mm); 
    \end{tikzpicture}
     = \frac{1}{4} 
    \begin{tikzpicture}[anchorbase,smallnodes,scale=.2] 
            
            \fill[rd] (-0.5,0.5) circle (2.5mm); 
            \fill[rd] (-0.5,-0.5) circle (2.5mm); 
            \draw[rd] (-0.5,0.5) \pu (-0.5,-0.5);
            
            \fill[gr] (0.5,0.5) circle (2.5mm); 
            \fill[gr] (0.5,-0.5) circle (2.5mm); 
            \draw[gr] (0.5,0.5) \pu (0.5,-0.5);
            
    \end{tikzpicture}
\end{equation*}
\end{example}

This example, without the polynomial correction, had been found by Klein, \cite[Example 5.2.10]{K14}, although no explicit maps were given. This example is also explored in \cite[Example 5E.17]{ERT25}, although in a slightly different context. 

\begin{example}
\label{ex:ustsu}
Even though $ustsu$ is not a Deodhar element, one can compute that $C_uC_{sts}C_u = C_{ustsu}$, so in terms of Soergel bimodules $B_{u} \otimes B_{sts} \otimes B_{u} \cong B_{ustsu}$. The projection map $B_{s} \otimes B_{t} \otimes B_{s} \to B_{sts}$ is well-known and the maps $\mu_{ustsu}$ and $\delta_{ustsu}$ can be constructed similarly as in Example \ref{ex:tsut}. In diagrammatic terms,
\begin{equation*}
    \mu_{ustsu} = 
    \begin{tikzpicture}[anchorbase,smallnodes,scale=.25] 
        \draw[rd, line cap = round]   (0,0) \pu (2,4)
                        (7,0) \pu (5,4)
                        (2,0) \ur (3.5,1) \rd (5,0)
                        (3.5,1) \pu (3.5,2)
                        (3.5,2) \ulu (1.75,3)
                        (3.5,2) \uru (5.25,3)
                        ;

        \draw[bl, line cap = round]   (1,0) \uur (3.5,2)
                    (6,0) \uul (3.5,2)
                    (3.5,2)\pu (3.5,4) 
                    ;

        \draw[gr, line cap = round]   (-1,0) \pu (1,4)
                    (8,0) \pu (6,4) 
                    (3,0) \ur (3.5,0.5) \rd (4,0)
                    ;
    \end{tikzpicture}
    \qquad \qquad
    \delta_{ustsu} = 
        \begin{tikzpicture}[anchorbase,smallnodes,scale=.25]
        \draw[rd, line cap = round]       (2,-4) \pu (0,0)
                        (5,-4) \pu (7,0)
                        (5,0) \dl (3.5,-1) \lu (2,0)
                        (3.5,-1) \pu (3.5,-2) 
                        (1.75,-3) \uur (3.5,-2)
                        (5.25,-3) \uul (3.5,-2) 
                        ;
                        
        \draw[bl, line cap = round]   (3.5,-2) \ulu (1,0) 
                    (3.5,-2) \uru (6,0)
                    (3.5,-4) \pu (3.5,-2) 
                    ;

        \draw[gr, line cap = round]   (1,-4) \pu (-1,0)  
                    (6,-4) \pu (8,0)  
                    (4,0) \dl (3.5,-0.5) \lu (3,0)
                    ;
        \fill[pu, line cap = round] (3.5,0) circle (2mm); 
                        
    \end{tikzpicture}
    \quad \text{with} \quad \begin{tikzpicture}[anchorbase,smallnodes,scale=.3]
    \fill[pu] (3.5,0) circle (3mm);
    \end{tikzpicture} = \frac{1}{8}     \begin{tikzpicture}[anchorbase,smallnodes,scale=.15] 
    \fill[rd] (0,0) circle (3mm); 
    \fill[rd] (2,4) circle (3mm);
    \fill[rd] (0,4) circle (3mm);
    \fill[rd] (2,0) circle (3mm);
    \fill[bl] (1,4) circle (3mm);
    \fill[bl] (1,0) circle (3mm);

    \draw[rd]   (0,0) \uur (1,1)
                (2,0) \uul (1,1)
                (1,1) \pu (1,3)
                (1,3) \ulu (0,4)
                (1,3) \uru (2,4);
    \draw[bl]   (1,0) \pu (1,1)
                (1,1) \lu (0,2) \ur (1,3) \rd (2,2) \dl (1,1)
                (1,3) \pu (1,4);
    \draw[gr]   (1,-1) \lu (-2,2) \ur (1,5) \rd (4,2) \dl (1,-1);            
    \end{tikzpicture}.
\end{equation*}
This example allows one to find a Frobenius structure on $ustsu$, which cannot by previously developed methods, see \cite[Example 4.16]{MMMTZ23}. Notice that removing the polynomial correction from the map $\delta_{ustsu}$ also endows $B_{ustsu}$ with the structure of a Frobenius algebra object but \emph{not} with the structure of a 2-categorical idempotent, since without the correction, the composition of the comultiplication and multiplication maps is equal to zero. 
\end{example}

With $B_{tsut}$ and $B_{ustsu}$ we thus have 2-categorical idempotents in $\SBim_4$, and hence corresponding objects in $\kkSBim_4$, so \Cref{q:separable} has positive answers in these cases. Even though these $2$-categorical idempotents do not come directly from singular Soergel bimodules, \Cref{q:sepsequiv} has a positive answer unlike \Cref{q:sepsplits}: 

\begin{proposition}
\label{prop:morita}
In $\kkSBim_4$ we have equivalences of objects 
\[ B_{tsut}\simeq B_{su}\qquad \text{and} \qquad B_{ustsu}\simeq B_{sts}.\]
As a consequence, 2-categorical idempotents $B_{tsut}$ and $B_{ustsu}$ are in the essential image of $\sSbim_4\to \kkSBim_4$. 
\end{proposition}
\begin{proof}
    An easy computation shows that $B_{tsu}$ and $B_{sut}$ admit the structure of bimodules between the 2-categorical idempotents $B_{tsut}$ and $B_{su}$ which are, furthermore, mutually inverse up to 2-isomorphism. The case of $B_{ustsu}$ is analogous.
\end{proof}

A key aspect of $n=4$ is that the multiplication and section for the new examples of 2-categorical idempotents are simply obtained from old ones, namely those corresponding to longest elements of parabolics, by \emph{sandwiching}: wrapping the diagrammatic maps for the longest element with appropriately coloured lines. Following the proof of \Cref{prop:morita}, this can be interpreted as saying that the 2-categorical idempotents $B_{tsut}$ and $B_{ustsu}$ were \emph{defined} to be (Morita) equivalent to their longest-element-of-parabolic cousins. 

Essential for this strategy is the existence of a sequence of simple transpositions such that, while staying in the same two-sided cell (in the sense of \cite{KL79}), one can conjugate in a length-decreasing way to the longest element of a parabolic. In general, the following is true.

\begin{proposition}
\label{prop:sandwich}
    Let $d\in \mathcal{D}\subset W$ be a distinguished involution. Suppose $w_J$ is the longest element of a parabolic, such that there exists a sequence $(s_1, \ldots s_k)$ of simple reflections with:
    \begin{enumerate}
        \item $d= s_k \cdots s_1w_J s_1 \cdots s_k$ and $\ell(d)=\ell(w_J)+2k$.
        \item For all $1 \leq i \leq k$, the element $s_i \cdots s_{1} w_J s_{1} \cdots s_i$ is in the same two-sided cell as $w_J$.
        \item $B_{s_k}\cdots B_{s_1}B_{w_J} B_{s_1} \cdots B_{s_k}$ is indecomposable. 
    \end{enumerate}
    
    Then $B_d$ inherits the structure of a Frobenius algebra in $\SBim_n$ from $B_{w_J}$.
\end{proposition}
\begin{proof}
    Under these very strict assumptions controlling the behaviour of the bimodules, the statement follows from the fact that $B_{w_J}$ admits a Frobenius algebra structure realized diagrammatically with the thick calculus from \cite{E16}. The structural morphismus for a Frobenius algebra structure on $B_d\cong B_{s_k}\cdots B_{s_1}B_{w_J} B_{s_1} \cdots B_{s_k}$ are then obtained diagrammatically by sandwiching those of $B_{w_J}$ by a \emph{rainbow of lines} labeled $s_1,\dots, s_k$, as in \Cref{ex:tsut} and \Cref{ex:ustsu} (ignoring the polynomial corrections). 
\end{proof}

The comultiplication of the Frobenius algebra structure on $B_d$ constructed in the proof of \Cref{prop:sandwich} is typically not a section to the multiplication. However, for longest elements of parabolics and for the $S_4$ examples above, there exists a polynomial correction to the comultiplication, given by a scaling of the composite of unit and counit, that yields a section to the multiplication. Whenever this holds, $B_d$ yields a 2-categorical idempotent that is (Morita) equivalent to $B_{w_J}$ in $\kkSBim$ and hence in the essential image of $\sSbim\to \kkSBim$. 

\begin{question} Suppose that $d\in \mathcal{D}\subset W$ satisfies the requirements of \Cref{prop:sandwich} with associated longest element of a parabolic $w_J$, so that $B_d$ can be equipped with the induced Frobenius algebra structure. Consider the polynomial correction to the comultiplication on $B_d$, with polynomial obtained by scaling the rainbow-circled composite of counit and unit of $B_{w_J}$. Is this corrected comultiplication a section to the multiplication?  
\end{question}
It would also be interesting if this proposed correction is a \emph{generalized barbell} from \cite[Definition 4.34]{EH17}.
\smallskip

Finally we comment on the severe limitations of \Cref{prop:sandwich}.

\begin{example}
    Condition (3) in \Cref{prop:sandwich} is rarely satisfied, typically the Bott--Samelson bimodule is decomposable with $B_d$ as one of its summands. Finding the correct projection onto this summand is a well-known challenge in the theory of Soergel bimodules. Already in $S_5$ one can find an example of this failure. Take $d = tustsvut$, the longest element $stsv$ and the sequence $(u,t)$. One can verify that this data satisfies conditions (1) and (2), but (3) fails:
    \[B_tB_uB_{stsv}B_uB_t \cong B_d \oplus B_{tuvsts} \oplus B_{vstust} \oplus B_{stsv}\]
\end{example}

\begin{example}
\label{exa:probone}
Involutions in symmetric groups are parametrized via the Robinson–Schensted correspondence by (pairs of identical) Young tableaux. The corresponding 2-sided cell can be read off by the shape of the tableaux. Two things are true:
\begin{itemize}
  \item The length-minimizing involutions in the 2-sided cells are exactly the longest elements of parabolics \cite[Theorem 1.1]{H05}.
    \item One can move between any two involutions of the same cycle type by iterated conjugation by simple transpositions.
\end{itemize}

\end{example}
Given any involution $d\in S_n$, one may thus hope to find a suitable $w_J$ in the same 2-sided cell (and thus equal cycle type) and then a sequence $(s_1, \ldots s_k)$ of simple transpositions, such that (1) and (2) in \Cref{prop:sandwich} are satisfied. Unfortunately, this is not the case:

\begin{example}
\label{exa:probtwo}
    For the involution \[d = s_5s_4s_3s_1s_8s_7s_6s_5s_4s_3s_2s_8s_7s_6s_5s_4s_3s_7s_6s_5s_4s_7s_6s_5s_8s_7s_8 \in S_9\]
     all length-decreasing conjugations by simple transpositions leave the 2-sided cell of $d$, so Condition (2) in \ref{prop:sandwich} is not satisfied. 
\end{example}

It would be interesting, whether the rainbow approach of \Cref{prop:sandwich} can be extended and made effective in a larger number of examples by using either the knowledge of concrete idempotents or the thick calculus for Soergel bimodules to overcome the limitations expressed in \Cref{exa:probone} and \Cref{exa:probtwo}.

\section{Higher idempotents and deformations of foam categories}

Our next goal is to establish the relation between the monoidal 2-categories of singular Bott--Samelson bimodules and foams.
\subsection{Foams}
\label{sec:foams}

We recall the (secretely monoidal) 2-category of foams defined in \cite{QR16} as setting for a combinatorial formulation of $\glN$ link homology, motivated by categorical skew Howe duality. For simplicity and later applications we work over the complex numbers $\C$.

\begin{definition}
\label{def:foams}
    Consider the locally $\C$-linear 2-category of progressive foams, $\Foam_+$ assembled from the following data:
    \begin{itemize}
        \item objects are finite (possibly empty) sequences $\mathbf{a} = (a_1, \ldots , a_k)$ with $a_i \in \N$;
        \item 1-morphisms are given by webs, generated under horizontal and vertical stacking, by identity webs and by merge and split trivalent vertices of the form 
        
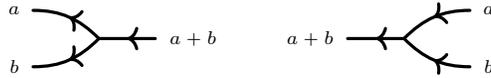
\[ \begin{tikzpicture}[line cap = round, scale=.5]
    	\draw [very thick,mid arrow] (2.25,0) to (.75,0);
    	\draw [very thick,mid arrow] (.75,0) to [out=135,in=0] (-1,.75);
    	\draw [very thick,mid arrow] (.75,0) to [out=225,in=0] (-1,-.75);
    	\node at (3.25,0) {\tiny $a+b$};
        \node at (-1.5,.75) {\tiny $a$};
        \node at (-1.5,-.75) {\tiny $b$};
    \end{tikzpicture} \qquad
    \begin{tikzpicture}[line cap = round, scale=.5]
	   \draw [very thick,mid arrow] (-.75,0) to (-2.25,0) ;
	   \draw [very thick,mid arrow] (1,.75) to [out=180,in=45] (-.75,0);
	   \draw [very thick,mid arrow] (1,.-.75) to [out=180,in=315] (-.75,0);
	   \node at (-3.25,0) {\tiny $a+b$};
	   \node at (1.5,.75) {\tiny $a$};
	   \node at (1.5,-.75) {\tiny $b$};
    \end{tikzpicture}
 \]
 which are considered as mapping from the sequence of labels on the right boundary to the sequence of labels on the left boundary.
 
    These webs are \emph{progressive} in the sense that they have no vertical tangencies and are always oriented from right to left. We omit the orientation in upcoming diagrams. 
        \item 2-morphisms are given by progressive \emph{foams}, generated under stacking operations in three different directions and linear combinations by the following basic foams, where we label the thickness of each facet:
        
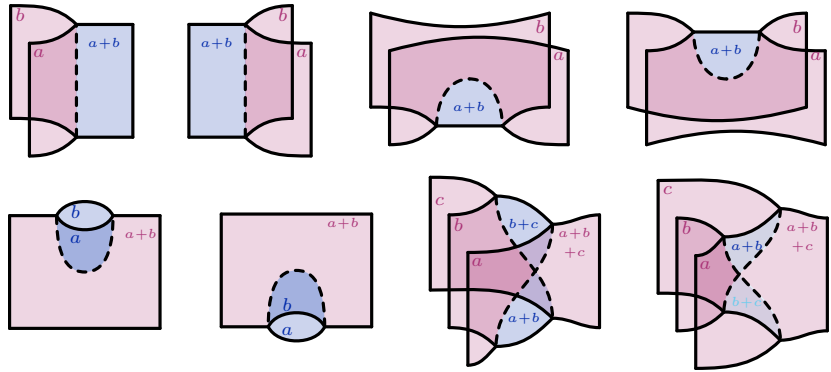
\[
            \begin{tikzpicture} [anchorbase, line cap = round, scale=.5,fill opacity=0.2]
	%shading
	\path[fill=cl_dark_blue] (2.25,3) to (.75,3) to (.75,0) to (2.25,0);
	\path[fill=cl_red] (.75,3) to [out=225,in=0] (-.5,2.5) to (-.5,-.5) to [out=0,in=225] (.75,0);
	\path[fill=cl_red] (.75,3) to [out=135,in=0] (-1,3.5) to (-1,.5) to [out=0,in=135] (.75,0);	
	%bottom web
	\draw [very thick] (2.25,0) to (.75,0);
	\draw [very thick] (.75,0) to [out=135,in=0] (-1,.5);
	\draw [very thick] (.75,0) to [out=225,in=0] (-.5,-.5);
	%seam
	\draw[very thick, black, dashed] (.75,0) to (.75,3);
	%vertical edges
	\draw [very thick] (2.25,3) to (2.25,0);
	\draw [very thick] (-1,3.5) to (-1,.5);
	\draw [very thick] (-.5,2.5) to (-.5,-.5);
	%top web
	\draw [very thick] (2.25,3) to (.75,3);
	\draw [very thick] (.75,3) to [out=135,in=0] (-1,3.5);
	\draw [very thick] (.75,3) to [out=225,in=0] (-.5,2.5);
	%labels
	\node [cl_dark_blue, opacity=1]  at (1.5,2.5) {\tiny{$_{a+b}$}};
	\node[cl_red, opacity=1] at (-.75,3.25) {\tiny{$b$}};
	\node[cl_red, opacity=1] at (-.25,2.25) {\tiny{$a$}};		
    \end{tikzpicture} \qquad
    \begin{tikzpicture} [anchorbase, line cap = round, scale=.5,fill opacity=0.2]
	%shading
	\path[fill=cl_dark_blue] (-2.25,3) to (-.75,3) to (-.75,0) to (-2.25,0);
	\path[fill=cl_red] (-.75,3) to [out=45,in=180] (.5,3.5) to (.5,.5) to [out=180,in=45] (-.75,0);
	\path[fill=cl_red] (-.75,3) to [out=315,in=180] (1,2.5) to (1,-.5) to [out=180,in=315] (-.75,0);	
	%bottom web
	\draw [very thick] (-2.25,0) to (-.75,0);
	\draw [very thick] (-.75,0) to [out=315,in=180] (1,-.5);
	\draw [very thick] (-.75,0) to [out=45,in=180] (.5,.5);
	%seam
	\draw[very thick, black, dashed] (-.75,0) to (-.75,3);
	%vertical edges
	\draw [very thick] (-2.25,3) to (-2.25,0);
	\draw [very thick] (1,2.5) to (1,-.5);
	\draw [very thick] (.5,3.5) to (.5,.5);
	%top web
	\draw [very thick] (-2.25,3) to (-.75,3);
	\draw [very thick] (-.75,3) to [out=315,in=180] (1,2.5);
	\draw [very thick] (-.75,3) to [out=45,in=180] (.5,3.5);
	%labels
	\node [cl_dark_blue, opacity=1]  at (-1.5,2.5) {\tiny{$_{a+b}$}};
	\node[cl_red, opacity=1] at (.25,3.25) {\tiny{$b$}};
	\node[cl_red, opacity=1] at (.75,2.25) {\tiny{$a$}};		
\end{tikzpicture} \qquad 
\begin{tikzpicture} [anchorbase, line cap = round, scale=.5,fill opacity=0.2]
	%shading
	\path [fill=cl_red] (4.25,-.5) to (4.25,2) to [out=165,in=15] (-.5,2) to (-.5,-.5) to 
	[out=0,in=225] (.75,0) to [out=90,in=180] (1.625,1.25) to [out=0,in=90] 
	(2.5,0) to [out=315,in=180] (4.25,-.5);
	\path [fill=cl_red] (3.75,.5) to (3.75,3) to [out=195,in=345] (-1,3) to (-1,.5) to 
	[out=0,in=135] (.75,0) to [out=90,in=180] (1.625,1.25) to [out=0,in=90] 
	(2.5,0) to [out=45,in=180] (3.75,.5);
	\path[fill=cl_dark_blue] (.75,0) to [out=90,in=180] (1.625,1.25) to [out=0,in=90] (2.5,0);
	%bottom web
	\draw [very thick] (2.5,0) to (.75,0);
	\draw [very thick] (.75,0) to [out=135,in=0] (-1,.5);
	\draw [very thick] (.75,0) to [out=225,in=0] (-.5,-.5);
	\draw [very thick] (3.75,.5) to [out=180,in=45] (2.5,0);
	\draw [very thick] (4.25,-.5) to [out=180,in=315] (2.5,0);
	%seam
	\draw [very thick, black, dashed] (.75,0) to [out=90,in=180] (1.625,1.25);
	\draw [very thick, black, dashed] (1.625,1.25) to [out=0,in=90] (2.5,0);
	%vertical edges
	\draw [very thick] (3.75,3) to (3.75,.5);
	\draw [very thick] (4.25,2) to (4.25,-.5);
	\draw [very thick] (-1,3) to (-1,.5);
	\draw [very thick] (-.5,2) to (-.5,-.5);
	%top web
	\draw [very thick] (4.25,2) to [out=165,in=15] (-.5,2);
	\draw [very thick] (3.75,3) to [out=195,in=345] (-1,3);
	%labels
	\node [cl_dark_blue, opacity=1]  at (1.625,.5) {\tiny{$_{a+b}$}};
	\node[cl_red, opacity=1] at (3.5,2.65) {\tiny{$b$}};
	\node[cl_red, opacity=1] at (4,1.85) {\tiny{$a$}};		
\end{tikzpicture} \qquad 
\begin{tikzpicture} [anchorbase, line cap = round, scale=.5,fill opacity=0.2]
	%shading
	\path [fill=cl_red] (4.25,2) to (4.25,-.5) to [out=165,in=15] (-.5,-.5) to (-.5,2) to
	[out=0,in=225] (.75,2.5) to [out=270,in=180] (1.625,1.25) to [out=0,in=270] 
	(2.5,2.5) to [out=315,in=180] (4.25,2);
	\path [fill=cl_red] (3.75,3) to (3.75,.5) to [out=195,in=345] (-1,.5) to (-1,3) to [out=0,in=135]
	(.75,2.5) to [out=270,in=180] (1.625,1.25) to [out=0,in=270] 
	(2.5,2.5) to [out=45,in=180] (3.75,3);
	\path[fill=cl_dark_blue] (2.5,2.5) to [out=270,in=0] (1.625,1.25) to [out=180,in=270] (.75,2.5);
	%bottom web
	\draw [very thick] (4.25,-.5) to [out=165,in=15] (-.5,-.5);
	\draw [very thick] (3.75,.5) to [out=195,in=345] (-1,.5);
	%seam
	\draw [very thick, black, dashed] (2.5,2.5) to [out=270,in=0] (1.625,1.25);
	\draw [very thick, black, dashed] (1.625,1.25) to [out=180,in=270] (.75,2.5);
	%vertical edges
	\draw [very thick] (3.75,3) to (3.75,.5);
	\draw [very thick] (4.25,2) to (4.25,-.5);
	\draw [very thick] (-1,3) to (-1,.5);
	\draw [very thick] (-.5,2) to (-.5,-.5);
	%top web
	\draw [very thick] (2.5,2.5) to (.75,2.5);
	\draw [very thick] (.75,2.5) to [out=135,in=0] (-1,3);
	\draw [very thick] (.75,2.5) to [out=225,in=0] (-.5,2);
	\draw [very thick] (3.75,3) to [out=180,in=45] (2.5,2.5);
	\draw [very thick] (4.25,2) to [out=180,in=315] (2.5,2.5);
	%labels
	\node [cl_dark_blue, opacity=1]  at (1.625,2) {\tiny{$_{a+b}$}};
	\node[cl_red, opacity=1] at (3.5,2.65) {\tiny{$b$}};
	\node[cl_red, opacity=1] at (4,1.85) {\tiny{$a$}};		
\end{tikzpicture}
        \]
        \[\begin{tikzpicture} [anchorbase, line cap = round, scale=.5,fill opacity=0.2]
	%back cup 
	\path[fill=cl_dark_blue] (-.75,4) to [out=270,in=180] (0,2.5) to [out=0,in=270] (.75,4) .. controls (.5,4.5) and (-.5,4.5) .. (-.75,4);
	%sheet
	\path[fill=cl_red] (-.75,4) to [out=270,in=180] (0,2.5) to [out=0,in=270] (.75,4) -- (2,4) -- (2,1) -- (-2,1) -- (-2,4) -- (-.75,4);
	%front cup
	\path[fill=cl_dark_blue] (-.75,4) to [out=270,in=180] (0,2.5) to [out=0,in=270] (.75,4) .. controls (.5,3.5) and (-.5,3.5) .. (-.75,4);
	%bottom web
	\draw[very thick] (2,1) -- (-2,1);
	\path (.75,1) .. controls (.5,.5) and (-.5,.5) .. (-.75,1); %for spacing symmetry
	%seam
	\draw [very thick, black, dashed] (-.75,4) to [out=270,in=180] (0,2.5) to [out=0,in=270] (.75,4);
	%vertical edges
	\draw[very thick] (2,4) -- (2,1);
	\draw[very thick] (-2,4) -- (-2,1);
	%top web
	\draw[very thick] (2,4) -- (.75,4);
	\draw[very thick] (-.75,4) -- (-2,4);
	\draw[very thick] (.75,4) .. controls (.5,3.5) and (-.5,3.5) .. (-.75,4);
	\draw[very thick] (.75,4) .. controls (.5,4.5) and (-.5,4.5) .. (-.75,4);
	%labels
	\node [cl_red, opacity=1]  at (1.5,3.5) {\tiny{$_{a+b}$}};
	\node[cl_dark_blue, opacity=1] at (-.25,3.375) {\tiny{$a$}};
	\node[cl_dark_blue, opacity=1] at (-.25,4.1) {\tiny{$b$}};	
\end{tikzpicture} \qquad
\begin{tikzpicture} [anchorbase, line cap = round, scale=.5,fill opacity=0.2]
	%back cup 
	\path[fill=cl_dark_blue] (-.75,-4) to [out=90,in=180] (0,-2.5) to [out=0,in=90] (.75,-4) .. controls (.5,-4.5) and (-.5,-4.5) .. (-.75,-4);
	%sheet
	\path[fill=cl_red] (-.75,-4) to [out=90,in=180] (0,-2.5) to [out=0,in=90] (.75,-4) -- (2,-4) -- (2,-1) -- (-2,-1) -- (-2,-4) -- (-.75,-4);
	%front cup
	\path[fill=cl_dark_blue] (-.75,-4) to [out=90,in=180] (0,-2.5) to [out=0,in=90] (.75,-4) .. controls (.5,-3.5) and (-.5,-3.5) .. (-.75,-4);
	%top web
	\draw[very thick] (2,-1) -- (-2,-1);
	\path (.75,-1) .. controls (.5,-.5) and (-.5,-.5) .. (-.75,-1); %for spacing symmetry
	%seam
	\draw [very thick, black, dashed] (.75,-4) to [out=90,in=0] (0,-2.5) to [out=180,in=90] (-.75,-4);
	%vertical edges
	\draw[very thick] (2,-4) -- (2,-1);
	\draw[very thick] (-2,-4) -- (-2,-1);
	%bottom web
	\draw[very thick] (2,-4) -- (.75,-4);
	\draw[very thick] (-.75,-4) -- (-2,-4);
	\draw[very thick] (.75,-4) .. controls (.5,-3.5) and (-.5,-3.5) .. (-.75,-4);
	\draw[very thick] (.75,-4) .. controls (.5,-4.5) and (-.5,-4.5) .. (-.75,-4);
	%labels
	\node [cl_red, opacity=1]  at (1.25,-1.25) {\tiny{$_{a+b}$}};
	\node[cl_dark_blue, opacity=1] at (-.25,-3.4) {\tiny{$b$}};
	\node[cl_dark_blue, opacity=1] at (-.25,-4.1) {\tiny{$a$}};
\end{tikzpicture} \qquad 
\begin{tikzpicture} [anchorbase, line cap = round, scale=.5,fill opacity=0.2]
	%shading
	%	\node[opacity=1] at (.29,1.5) {$+$};
	\path[fill=cl_red] (-2.5,4) to [out=0,in=135] (-.75,3.5) to [out=270,in=90] (.75,.25)
	to [out=135,in=0] (-2.5,1);
	\path[fill=cl_dark_blue] (-.75,3.5) to [out=270,in=125] (.29,1.5) to [out=55,in=270] (.75,2.75) 
	to [out=135,in=0] (-.75,3.5);
	\path[fill=cl_dark_blue] (-.75,-.5) to [out=90,in=235] (.29,1.5) to [out=315,in=90] (.75,.25) 
	to [out=225,in=0] (-.75,-.5);
	\path[fill=cl_red] (-2,3) to [out=0,in=225] (-.75,3.5) to [out=270,in=125] (.29,1.5)
	to [out=235,in=90] (-.75,-.5) to [out=135,in=0] (-2,0);
	\path[fill=cl_red] (-1.5,2) to [out=0,in=225] (.75,2.75) to [out=270,in=90] (-.75,-.5)
	to [out=225,in=0] (-1.5,-1);
	\path[fill=cl_red] (2,3) to [out=180,in=0] (.75,2.75) to [out=270,in=55] (.29,1.5)
	to [out=305,in=90] (.75,.25) to [out=0,in=180] (2,0);
	%bottom web
	\draw[very thick] (2,0) to [out=180,in=0] (.75,.25);
	\draw[very thick] (.75,.25) to [out=225,in=0] (-.75,-.5);
	\draw[very thick] (.75,.25) to [out=135,in=0] (-2.5,1);
	\draw[very thick] (-.75,-.5) to [out=135,in=0] (-2,0);
	\draw[very thick] (-.75,-.5) to [out=225,in=0] (-1.5,-1);
	%seams
	\draw[very thick, black, dashed] (-.75,3.5) to [out=270,in=90] (.75,.25);
	\draw[very thick, black, dashed] (.75,2.75) to [out=270,in=90] (-.75,-.5);	
	%vertical edges
	\draw[very thick] (-1.5,-1) -- (-1.5,2);	
	\draw[very thick] (-2,0) -- (-2,3);
	\draw[very thick] (-2.5,1) -- (-2.5,4);	
	\draw[very thick] (2,3) -- (2,0);
	% top web
	\draw[very thick] (2,3) to [out=180,in=0] (.75,2.75);
	\draw[very thick] (.75,2.75) to [out=135,in=0] (-.75,3.5);
	\draw[very thick] (.75,2.75) to [out=225,in=0] (-1.5,2);
	\draw[very thick]  (-.75,3.5) to [out=225,in=0] (-2,3);
	\draw[very thick]  (-.75,3.5) to [out=135,in=0] (-2.5,4);
	%labels
	\node[cl_red, opacity=1] at (-2.25,3.375) {\tiny$c$};
	\node[cl_red, opacity=1] at (-1.75,2.75) {\tiny$b$};	
	\node[cl_red, opacity=1] at (-1.25,1.75) {\tiny$a$};
	\node[cl_dark_blue, opacity=1] at (0,2.75) {\tiny$_{b+c}$};
	\node[cl_dark_blue, opacity=1] at (0,.25) {\tiny$_{a+b}$};
	\node[cl_red, opacity=1] at (1.35,2.5) {\tiny$_{a+b}$};	
	\node[cl_red, opacity=1] at (1.35,2) {\tiny$_{+c}$};	
\end{tikzpicture} \qquad 
\begin{tikzpicture} [anchorbase, line cap = round, scale=.5,fill opacity=0.2]
	%shading
	%	\node[opacity=1] at (-.35,1.5) {$+$};
	\path[fill=cl_red] (-2.5,4) to [out=0,in=135] (.75,3.25) to [out=270,in=90] (-.75,.5)
	to [out=135,in=0] (-2.5,1);
	\path[fill=cl_blue] (-.75,2.5) to [out=270,in=125] (-.35,1.5) to [out=45,in=270] (.75,3.25) 
	to [out=225,in=0] (-.75,2.5);
	\path[fill=cl_blue] (-.75,.5) to [out=90,in=235] (-.35,1.5) to [out=315,in=90] (.75,-.25) 
	to [out=135,in=0] (-.75,.5);	
	\path[fill=cl_red] (-2,3) to [out=0,in=135] (-.75,2.5) to [out=270,in=125] (-.35,1.5) 
	to [out=235,in=90] (-.75,.5) to [out=225,in=0] (-2,0);
	\path[fill=cl_red] (-1.5,2) to [out=0,in=225] (-.75,2.5) to [out=270,in=90] (.75,-.25)
	to [out=225,in=0] (-1.5,-1);
	\path[fill=cl_red] (2,3) to [out=180,in=0] (.75,3.25) to [out=270,in=45] (-.35,1.5) 
	to [out=315,in=90] (.75,-.25) to [out=0,in=180] (2,0);				
	%bottom web
	\draw[very thick] (2,0) to [out=180,in=0] (.75,-.25);
	\draw[very thick] (.75,-.25) to [out=135,in=0] (-.75,.5);
	\draw[very thick] (.75,-.25) to [out=225,in=0] (-1.5,-1);
	\draw[very thick]  (-.75,.5) to [out=225,in=0] (-2,0);
	\draw[very thick]  (-.75,.5) to [out=135,in=0] (-2.5,1);	
	%seams
	\draw[very thick, black, dashed] (-.75,2.5) to [out=270,in=90] (.75,-.25);
	\draw[very thick, black, dashed] (.75,3.25) to [out=270,in=90] (-.75,.5);
	%vertical edges
	\draw[very thick] (-1.5,-1) -- (-1.5,2);	
	\draw[very thick] (-2,0) -- (-2,3);
	\draw[very thick] (-2.5,1) -- (-2.5,4);	
	\draw[very thick] (2,3) -- (2,0);
	% top web
	\draw[very thick] (2,3) to [out=180,in=0] (.75,3.25);
	\draw[very thick] (.75,3.25) to [out=225,in=0] (-.75,2.5);
	\draw[very thick] (.75,3.25) to [out=135,in=0] (-2.5,4);
	\draw[very thick] (-.75,2.5) to [out=135,in=0] (-2,3);
	\draw[very thick] (-.75,2.5) to [out=225,in=0] (-1.5,2);
	%labels
	\node[cl_red, opacity=1] at (-2.25,3.75) {\tiny$c$};
	\node[cl_red, opacity=1] at (-1.75,2.75) {\tiny$b$};	
	\node[cl_red, opacity=1] at (-1.25,1.75) {\tiny$a$};
	\node[cl_dark_blue, opacity=1] at (-.125,2.25) {\tiny$_{a+b}$};
	\node[cl_blue, opacity=1] at (-.125,.75) {\tiny$_{b+c}$};
	\node[cl_red, opacity=1] at (1.35,2.75) {\tiny$_{a+b}$};
	\node[cl_red, opacity=1] at (1.35,2.25) {\tiny$_{+c}$};	
\end{tikzpicture}\]
Foams are considered as 2-morphisms from the web appearing on the bottom to the web appearing at the top. Facets of foams with thickness $a$ can be decorated by symmetric polynomials in $a$-variables.
Foams are considered up to (certain) isotopies that preserve the progressiveness of webs obtained as level-sets and a list of further local linear relations as in \cite[Definition 3.1]{QR16}. For now, the colouring of the foams is only for aesthetic purpose, but later colours will be used to encode certain decorations by idempotents. 
    \end{itemize}
    The $\glN$ foam category $\NFoam_+$ is defined as the quotient of $\Foam_+$ by the ideal of 2-morphisms generated by foams that contain a facet with label larger than $N$. We denote by $\addFoam_+$ and $\addNFoam_+$ the local additive completions of $\Foam_+$ and $\NFoam_+$ respectively.
\end{definition}

     As in \cite{RW16}, we ignore the grading \cite[Definition 3.3]{QR16} on $\NFoam_+$ since it is not preserved when passing to the deformed version of the category. Sometimes it is useful to allow web edges and foam facets labeled by $0$, which should be considered as being erased. Polynomials in zero variables on such facets act as global scaling.

\begin{remark}
    For a natural number $n\in \N_0$, let $\Comp(n)$ denote the set of compositions (unordered partitions) of $n$, equipped with the partial order given by refining compositions. If $n\neq 0$, then its coarsest composition is $(n)$ and its finest one is $(1,\dots,1)$ with $n$ entries equal to $1$. The only composition of $0$ is the empty composition. The union $\Comp:=\bigsqcup_{n\in\N_0} \Comp(n)$ is again partially ordered and carries a monoidal structure given by concatenating compositions. The monoidal structure is $\N_0$-graded in the sense that compositions of $m,n\in \N_0$ concatenate to form a composition of $m+n\in \N_0$. 

    One possible perspective on foams, exploited in \cite{DW25}, is to interpret them as graphical language for monoidal $\infty$-functors out of $\Comp$ into monoidal $(2,2)$-categories, such that the images of all $1$-morphisms form a coherent system of ambidextrous adjunctions. Indeed, the objects of $\Comp$ are in bijection with the objects of any of the 2-categories from \Cref{def:foams}, refinements of compositions correspond to splitting webs, both their left and right adjoints correspond to merging webs (up to, perhaps, a grading shift autoequivalence), the zip, unzip, digon opening and digon closing foams correspond to the units and counits of these adjunctions, and the remaining foam generators are witnesses of the associativity of refinement in $\Comp$. 

    In \cite{DW25}, such structures occur in the context of factorizing families of perverse schobers of Coxeter type $A$. A prototypical example is given by the \emph{Soergel cube} of ambixdextrous adjunctions involving induction and restriction functors between partially symmetric polynomial rings, i.e. singular Bott-Samelson bimodules. We return to this relationship in \Cref{prop:foambsbim}.
\end{remark}

\begin{conv}
    Let $\Sigma$ be a multiset of $N$ complex numbers. We write $e_i(\Sigma)$ for the $i$th elementary symmetric polynomial in $\Sigma$, the signed $(N-i)$-th coefficient of the monic degree $N$ complex polynomial with root multiset $\Sigma$:
    \[
    P_{\Sigma}(X):=\prod_{\lambda\in \Sigma} (x-\lambda) = \sum_{i=0}^N (-1)^{N-i} e_{N-i}(\Sigma) X^i 
    \]
\end{conv}

\begin{definition}[{\cite[Section 4]{RW16}}]  
\label{def:deformed_foams}
    Consider the ideal of 2-morphisms in $\addNFoam_{+}$  generated by the following relation on 1-labeled foam facets
    \[ 0 = \sum_{i = 0}^{N} (-1)^{N-i-1}e_{N-i}(\Sigma) \isheet \]
    where a dot with a label $i$ denotes the polynomial $x^i$, which is symmetric in a single variable. Equivalently, the relation can be expressed as $P_{\Sigma}(\bullet)=0$, where $\bullet$ denotes the operator represented by the 1-labeled sheet decorated by $x$.
    We obtain the \emph{deformed foam 2-category} $\defNFoam$ as quotient of  $\addNFoam_{+}$ by this ideal. %In case $P_\Sigma(X)=X^N$ we write $\dotNFoam:=\defNFoam$.
\end{definition}

\begin{proposition}
    \label{prop:monoidal_foams}
    The 2-categories $\Foam_+$, $\addFoam_+$, $\NFoam_+$, $\addNFoam_+$, and $\defNFoam$ admit the structure of semistrict monoidal 2-categories, effectively by definition. More explicitly:
    \begin{itemize}
        \item The empty sequence serves as the monoidal unit. 
        \item For two objects $\mathbf{a} = (a_1, \ldots, a_k)$ and $ \mathbf{b} = (b_1, \ldots , b_l)$, their tensor product $\mathbf{a} \boxtimes \mathbf{b}$ is the sequence $(a_1 \ldots a_k, b_1, \ldots , b_l)$. 
        \item For an object $\mathbf{a}$ and a web $ \mathbf{c} \ot \mathbf{b} \colon W$, their tensor product $\mathbf{a} \boxtimes \mathbf{c} \ot \mathbf{a} \boxtimes \mathbf{b} \colon \mathbf{a} \boxtimes W$  is the vertical stacking of $W$ and the identity web $\id_{\mathbf{a}}$. One can analogously define the web $W \boxtimes \mathbf{a}$.
        \item For a foam $F \colon W_1 \to W_2$ between webs $\mathbf{c} \ot \mathbf{b} \colon W_1, W_2 $ and a sequence $\mathbf{a}$, the tensor product $\mathbf{a} \boxtimes F \colon \mathbf{a} \boxtimes W_1 \to \mathbf{a} \boxtimes W_2$ as the foam stacking $F$ with $\id_{\id_{\mathbf{a}}}$. One analogously defines the foam $F \boxtimes \mathbf{a}$. 
        \item The 2-isomorphism from Condition (7) in \cite{BN20} is the interchanger foam, represented below for arbitrary webs $\mathbf{a'} \ot \mathbf{a} \colon W$ and $\mathbf{b'} \ot \mathbf{b} \colon V$:
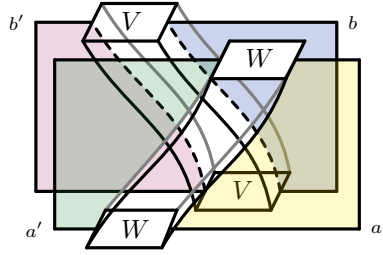
        \[
        \begin{tikzpicture}[line cap = round, scale=.5]

%%back foam

%%filling 

\fill[cl_red, fill opacity = .2](2.75,-4.5) to [out=100,in=300] (-0.25,0) to (-1.75,0) to (-1.75,-4.5) to (2.75,-4.5);
\fill[cl_dark_blue,fill opacity = .2] (4.75,-4.5) to [out=100,in=300] (1.75,0) to (6.25, 0) to (6.25, 0) to (6.25, -4.5); 

%tube 1 
 \draw [very thick] (2.5,-5) to [out=100,in=300] (-.5,-.5);
 \draw [very thick] (4.5,-5) to [out=100,in=300] (1.5,-.5);
 \draw [very thick, gray] (5,-4) to [out=100,in=300] (2,.5);
 \draw [very thick, gray] (3,-4) to [out=100,in=300] (0,.5);
\draw [very thick, dashed] (4.75,-4.5) to [out=100,in=300] (1.75,0);
\draw [very thick, dashed] (2.75,-4.5) to [out=100,in=300] (-.25,0);

%% top web
\draw [very thick] (-0.5,-.5) to (1.5,-.5);
\draw [very thick] (0,.5) to (2,.5);
\draw [very thick] (-0.5,-.5) to (0,.5);
\draw [very thick] (1.5,-.5) to (2,.5);
\draw [very thick] (-1.75,0) to (-.25,0);
\draw [very thick] (1.75,0) to (6.25, 0);

%% bottom web
\draw [very thick] (2.5,-5) to (4.5,-5);
\draw [very thick] (3,-4) to (5,-4);
\draw [very thick] (2.5,-5) to (3,-4);
\draw [very thick] (4.5,-5) to (5,-4);
\draw [very thick] (-1.75,-4.5) to (2.75,-4.5);
\draw [very thick] (4.75,-4.5) to (6.25, -4.5);

%%vertical edges 
\draw [very thick](-1.75,0) to (-1.75,-4.5);
\draw [very thick](6.25, -4.5) to  (6.25, 0);

%%lables and nodes
\node at (6.25, 0)[inner sep=.2pt, label=right:\tiny{$b$}](btop){};
\node at (-1.75,0)[inner sep=.2pt, label=left:\tiny{$b'$}](b'top){}; 
\node at (.8,-.5)[inner sep=.2pt, label=\small{$V$}](Vtop){}; 
\node at (3.75,-5)[inner sep=.2pt, label=\small{$V$}](Vbot){};

%% front foam

%%filling 
\fill[cl_green, fill opacity = .2](-0.15,-5.5) to [out=60,in=245] (3.15,-1) to (-1.25,-1) to (-1.25,-5.5) to (-0.15,-5.5);
\fill[cl_yellow, fill opacity = .2] (1.85,-5.5) to [out=50,in=251](5.15,-1) to (6.85, -1) to (6.85, -5.5) to (1.85,-5.5);
\fill[white] (2.9,-1.5) to (4.9,-1.5) to (5.4,-.5) to (3.4,-.5) to (2.9,-1.5);

%%tube
\draw [very thick, gray] (0.1,-5) to [out=50,in =245] (3.4,-.5) ;
\draw [very thick, gray] (2,-5) to [out=50,in=245](5.4,-.5);
\draw [very thick]  (-.4,-6) to [out=60,in=250] (2.9,-1.5);
\draw [very thick] (1.6,-6) to [out=60,in=250](4.9,-1.5);
\draw [very thick, opacity = 0]  (-0.15,-5.5) to [out=60,in=250] (3.15,-1);
\draw [very thick, opacity = 0] (1.85,-5.5) to [out=60,in=250](5.15,-1);

%%top web
\draw [very thick] (2.9,-1.5) to (4.9,-1.5);
\draw [very thick] (3.4,-.5) to (5.4,-.5);
\draw [very thick] (2.9,-1.5) to (3.4,-.5);
\draw [very thick] (4.9,-1.5) to (5.4,-.5);
\draw [very thick] (-1.25,-1) to (3.15,-1);
\draw [very thick] (5.15,-1) to (6.85, -1);

%%bottom web
\draw [very thick] (-0.4,-6) to (1.6,-6);
\draw [very thick] (0.1,-5) to (2,-5);
\draw [very thick] (-.4,-6) to (0.1,-5);
\draw [very thick] (1.6,-6) to (2,-5);
\draw [very thick] (-1.25,-5.5) to (-0.15,-5.5);
\draw [very thick] (1.85,-5.5) to (6.85, -5.5);

%%vertical edges 
\draw [very thick](-1.25,-1) to (-1.25,-5.5);
\draw [very thick](6.85, -5.5) to  (6.85, -1);

%lables and nodes
\node at (6.85, -5.5)[inner sep=.2pt, label=right:\tiny{$a$}](btop){};
\node at (-1.25,-5.5)[inner sep=.2pt, label=left:\tiny{$a'$}](b'top){}; 
\node at (4.2,-1.5)[inner sep=.2pt, label=\small{$W$}](Wtop){}; 
\node at (.9,-6.05)[inner sep=.2pt, label=\small{$W$}](Wbot){}; 
\end{tikzpicture} \]
    \end{itemize}
\end{proposition}

\begin{proof}
The data above satisfies the usual conditions \cite[Lemma 4]{BN20}, making $\Foam_+$ into a semistrict monoidal 2-category. Local additive completion preserves monoidal structures on 2-categories \cite[Remark 3.17]{SW24}, so $\addFoam_+$ inherits the structure. Finally, since the rest of the 2-categories are obtained as quotients of these previous two, they again inherit a semistrict monoidal structure. 
\end{proof}

\begin{proposition}
\label{prop:foambsbim}
The locally $\C$-linear monoidal 2-functor $\Foam_+ \to \msbsbim$ (for $k=\C$) sending
\begin{itemize}
    \item each composition $(a_1, \ldots , a_k)$ of $n\in \N_0$ to the graded ring $\C[x_1,\dots,x_n]^{S_{a_1}\times \cdots \times S_{a_k}}$;
    \item splitting trivalent vertices to graded induction bimodules between partially symmetric polynomial rings and merging trivalent vertices to the grading shifted restriction bimodules as in \Cref{lem:indresadj};
    \item the generating foams displayed in \Cref{def:foams} to identity morphisms, unit and counit morphisms of the ambidextrous adjunctions from \Cref{lem:indresadj}, and associativity morphisms respectively
\end{itemize}
is an equivalence of semistrict monoidal 2-categories.
\end{proposition}
\begin{proof}
    An explicit check of that the proposed mapping respects the defining relations of $\Foam_+$ is provided in \cite[\S 4.4]{RW16} as step along the way to a 2-functor to a certain 2-category of matrix factorizations. The factorization through singular Bott-Samelson bimodules is documented in \cite[Remark 3.3]{MR3982970}. The compatibility with the monoidal structure and bijectivity on the level of objects and 1-morphisms is clear from the constructions. The fact that the mapping induces isomorphisms on the vector spaces of 2-morphisms follows from work of Webster~\cite{MR3682839}, where Bott-Samelson bimodules are incarnated in terms of equivariant flag categories and, further, as quotients of Khovanov-Lauda categorified quantum groups that match the definition of $\NFoam_+$ in \cite{QR16}. See also \cite[Appendix A]{HRW21}. For a diagrammatic description of the functor we refer to \cite[Section 2.4]{BPRW25}. 
\end{proof}

Combining the equivalence from \Cref{prop:foambsbim} with the relevant quotient and completion 2-functors yields the following corollary.

\begin{corollary}
    \label{cor:2functors}
    We obtain locally $\C$-linear monoidal 2-functors:
    \begin{align*}
     \msbsbim &\to \Foam_+ \to \NFoam_+ \to \defNFoam\\
     \mssbim &\to \kar\addFoam_+\to \kar\addNFoam_+\to \kar \defNFoam\\
    \end{align*}
\end{corollary}

\subsection{Deformations}
Let $\X$ and $\B$ be finite sets of formal variables. We write $\Sym(\X)$ for the ring of symmetric polynomials in $\X$ (with integer coefficients) and consider
\[
H(\X,t) := \sum_{n\geq 0} h_n(\X) t^n := \prod_{x\in\X} (1- tx)^{-1}, \quad E(\X,t):=\sum_{n\geq 0} e_n(\X) t^n :=\prod_{x\in\X} (1+x t),
\]
the generating functions of the complete symmetric polynomials $h_n(\X)\in \Sym(\X)$ and elementary symmetric polynomials $e_n(\X)\in \Sym(\X)$, respectively. Furthermore we write $\Sym(\X|\B)$ for the ring of polynomials in $\X\sqcup \B$ which are separately symmetric in permutations of $\X$ and permutations of $\B$. The \emph{complete symmetric polynomials} $h_n(\X-\B)$ \emph{in the difference of alphabets} $\X-\B$ are given by the generating function:
\[
\sum_{n\geq 0}h_n(\X-\B) t^n := H(\X-\B,t):= H(\X,t)H(\B,t)^{-1}=H(\X,t)E(\B,-t)
\]
yielding the expression $h_n(\X-\B)=\sum_{i=0}^n h_{n-i}(\X)(-1)^i e_i(\B)$.
For $k,N\in \N$ let $\X,\B$ be sets of variables of degree $2$ of size $|X|=k$ and $|B|=N$. 

\begin{example}[{\cite[Section 5.4]{W13}}, {\cite[Gr 5]{LI}}, {\cite[Proposition 4.5.1]{AF23}}]
    Let $0\leq k\leq N\in \N$ and consider the Grassmannian $\mathrm{Gr}(k,N)$ of $k$-dimensional complex subspaces of $\C^N$ with its natural action of the general linear group $\GLN:=\GLN(\C)$. Then the $\GLN$-equivariant cohomology ring of $\mathrm{Gr}(k,N)$ admits the presentation
    \[
H^*_{\GLN}\left(\mathrm{Gr}(k,N)\right) = \frac{\Sym(\X|\B)}{\la h_{N-k+i}(\X-\B)|i>0 \ra }
    \]
    as a module of the equivariant cohomology of the point $H^*_{\GLN}(*)=\Sym(\B)$. Here $\X$ and $\B$ are disjoint sets of variables of degree $2$ of sizes $|\X|=k$ and $|\B|=N$. The ordinary cohomology ring is obtained by setting all the variables in $\B$ to zero.
\end{example}

As in \Cref{sec:foams} we fix a (multi)set $\Sigma$ of complex numbers with $|\Sigma|=N$.

\begin{definition}   For $0\leq k\leq N$ let $\X$ denote a set of degree two variables of size $|\X|=k$ and $\Sym_\C(\X)$ the $\C$-algebra of symmetric polynomials in $\X$. The $\Sigma$-deformed cohomology ring of the Grassmannian $\mathrm{Gr}(k,N)$ as 
    \[ H_{k}^\Sigma := \frac{\Sym_\C(\X)}{\la h_{N-k+i}(\X-\Sigma)|i>0 \ra }.\]
    Although the defining ideal is not homogeneous, we can consider $H_{k}^\Sigma$ as a filtered algebra with associated graded given by the cohomology $\C$-algebra of the Grassmannian.
\end{definition}

From now on, we use the following notation related to the multiset $\Sigma$.
\begin{conv}
\label{conv:sigma}
    Let $\lambda_1, \ldots \lambda_l$ be pairwise distinct complex numbers and $N_1, \ldots, N_l$ natural numbers such that $\sum_{i = 1}^{l}N_i = N$ and $\Sigma = \{ \lambda_1^{N_1}, \ldots , \lambda_l^{N_l} \}$ is the multiset containing $\lambda_i$ exactly $N_i$ times. We let $\powerset(\Sigma)$ denote the set of multisubsets of $\Sigma$. An element $A\in \powerset(\Sigma)$ can also be described by the multiplicities $0\leq k_i\leq N_i$, by which the elements $\lambda_i$ appear in $A$. In particular we can write $A=\{\lambda_1^{k_1}, \ldots , \lambda_l^{k_l}\}$. Furthermore, we define $A_i := \{\lambda_i^{k_i}\}$ for $1\leq i \leq l$, so that $A = \biguplus_i A_i$, where we denote the multiset sum by $\uplus$. If two multisubsets $A, B \subset \Sigma$ are such that they have no roots in common, we say that $A$ and $B$ are \emph{disjoint}.
\end{conv}

\begin{theorem}[{\cite[Theorem 3.11]{RW16}}] 
\label{thm:directsum}
Let $\Sigma$ 
be as in \Cref{conv:sigma} and $0\leq k \leq N=|\Sigma|$. Then we have a canonical decomposition
\begin{equation}
\label{eq:Hsumdecomp}
H_k^{\Sigma} \cong \bigoplus_{\substack{A\in \powerset(\Sigma)\\ |A|=k }} H^\Sigma_A
\end{equation}
of $\C$-algebras. 
For $A=\{\lambda_1^{k_1}, \ldots , \lambda_l^{k_l}\}$ the summands $H_A^\Sigma$ admit a further decomposition
\begin{equation}
\label{eq:Htensordecomp}
H^\Sigma_A\cong \bigotimes_{i} H_{k_i}^{N_i}
\end{equation}
 as $\C$-algebras, so that by combining we get 
\begin{equation}
\label{eq:decomp}
    H_k^{\Sigma} \cong \bigoplus_{\sum k_i = k}
     \bigotimes_{i} H_{k_i}^{N_i}.
     \end{equation}
\end{theorem} 

\subsection{Graphical language for 1-categorical idempotent completion}

\begin{lemma}[{\cite[Lemma 4.2]{RW16}}]
\label{lem:End-of-edge}
For $0\leq k\leq N$ we consider a $k$-labeled edge in $\NFoam_+^{\Sigma}$, i.e.  the identity endomorphism $\id_k$ of the object $k$. Then the operation of decorating the identity foam on $\id_k$ by symmetric polynomials in $\Sym_\C(\X)$ with $|\X|=k$ descends to an algebra isomorphism
\[
H_{k}^\Sigma \;\;\xrightarrow{\cong}\;\; \End_{\NFoam_+^{\Sigma}} (\id_k)
\]
\end{lemma}

\begin{corollary}
    \label{lem:tensor_end_ring}
    Let $\mathbf{a} = (a_1, \ldots , a_l) \in \defNFoam$. Then the natural algebra homomorphism
    \[\bigotimes_{i = 1}^{l} H_{a_i}^{\Sigma}\to\End_{\NFoam_+^{\Sigma}} (\id_\mathbf{a}) \]
    is an isomorphism. 
\end{corollary}

\begin{proof}
    By \cite[Theorem 4.29]{RW16}, every web in  $\NFoam_+$ and $\defNFoam$ corresponds to a matrix factorization as in \cite{W13}. The lemma then follows from
    \cite[(4-19)]{RW16} and  \Cref{lem:End-of-edge}.
\end{proof}

\begin{corollary}[{\cite[Corollary 3.17]{RW16}}]
\label{cor:orth_idemp}
There exists a complete collection of central orthogonal idempotents $\idem_A\in H_k^\Sigma$, which are indexed by multisubsets $A\subset \Sigma$ and project onto the direct summands of $H_k^\Sigma$. Any $k$-labeled facet of a foam in $\NFoam_+^{\Sigma}$ can be decorated by $\idem_A$. In particular, the identity foam facet can be expressed as a sum of $\idem_A$-decorated foams. Graphically, \[ \begin{tikzpicture}[anchorbase,line cap = round, scale=.5]
    %webs and labled points
    \draw[very thick] (0,0) to (0,2);
    \draw[very thick] (0,2) to (2,2);
    \draw[very thick] (2,2) to (2,0);
    \draw[very thick] (2,0) to (0,0);
    \node at (0,0)[ label=left:\tiny{$k$}](){};
    \node at (0,2)[label=left:\tiny{$k$}](){};
\end{tikzpicture} = \sum_{\substack{A\in \powerset(\Sigma)\\ |A|=k }} \begin{tikzpicture}[anchorbase,line cap = round, scale=.5]
    %webs and labled points
    \draw[very thick] (0,0) to (0,2);
    \draw[very thick] (0,2) to (2,2);
    \draw[very thick] (2,2) to (2,0);
    \draw[very thick] (2,0) to (0,0);
    \node at (0,0)[ label=left:\tiny{$k$}](){};
    \node at (0,2)[label=left:\tiny{$k$}](){};
    \node at (1,1){\tiny{$\idem_A$}}; 
\end{tikzpicture} \]  
\end{corollary}

\begin{definition}
Let $F$ be a foam in $\NFoam_+^{\Sigma}$. A \emph{colouring} of $F$ is a choice of idempotent $\idem_{A_f}$ as in \Cref{cor:orth_idemp} for every facet $f$ of $F$. Placing these extra decorations on all facets, in addition to possibly pre-existing decorations on the facets, results in the corresponding \emph{coloured foam}.

A colouring is called \emph{admissible} if the following flow condition is satisfied for every seam in $F$: Let $A$ (red) and $B$ (blue) denote the multisubsets colouring the two adjacent facets of non-maximal label and $C$ (green) the multisubset colouring the adjacent facet with maximal label, then $A \uplus B = C$.
    \[
    \begin{tikzpicture} [scale=.5,fill opacity=0.2]
    %filling 
    \path[fill=cl_green] (2.25,3) to (.75,3) to (.75,0) to (2.25,0);
    \path[fill=cl_dark_blue] (.75,3) to [out=225,in=0] (-.5,2.5) to (-.5,-.5) to [out=0,in=225] (.75,0);
    \path[fill=cl_red] (.75,3) to [out=135,in=0] (-1,3.5) to (-1,.5) to [out=0,in=135] (.75,0);	
	%bottom web
	\draw [very thick, line cap = round] (2.25,0) to (.75,0);
	\draw [very thick, line cap = round] (.75,0) to [out=135,in=0] (-1,.5);
	\draw [very thick,line cap = round] (.75,0) to [out=225,in=0] (-.5,-.5);
	%seam
	\draw[very thick, line cap = round, dashed] (.75,0) to (.75,3);
	%vertical edges
	\draw [very thick] (2.25,3) to (2.25,0);
	\draw [very thick] (-1,3.5) to (-1,.5);
	\draw [very thick] (-.5,2.5) to (-.5,-.5);
	%top web
	\draw [very thick, line cap = round] (2.25,3) to (.75,3);
	\draw [very thick, line cap = round] (.75,3) to [out=135,in=0] (-1,3.5);
	\draw [very thick, line cap = round] (.75,3) to [out=225,in=0] (-.5,2.5);
	%labels
	\node [black, opacity=1]  at (1.5,2.5){};
	\node[black, opacity=1] at (-.55,3.15) {};
	\node[black, opacity=1] at (-.05,2.15) {};
    \end{tikzpicture}
    \]
\end{definition}

By \Cref{cor:orth_idemp}, every foam in $\NFoam_+^{\Sigma}$ decomposes as a sum of all its possible colourings. It follows from \cite[Proposition 3.18]{RW16} that the coloured foam represents the zero 2-morphism in $\NFoam_+^{\Sigma}$ if the flow condition is violated at some seam, i.e. the colouring there satisfies $A \uplus B \neq C$. As a consequence, we obtain:

\begin{lemma}
    Let $F$ be a foam in $\NFoam_+^{\Sigma}$. Then, $F$ decomposes as a sum of all its admissibly coloured versions. 
\end{lemma}

\begin{conv}[Graphical language in the 1-categorical idempotent completion]
    \label{conv:col_web}
Consider a $k$-labeled web edge in $\defNFoam$, e.g. the identity 1-morphism on the object $(k)$. Then any idempotent $\idem_A$ from \Cref{cor:orth_idemp} defines an idempotent 2-morphism $F_A$ on the edge, here represented by a (green) coloured sheet. Inspired by \cite{BNM06} we use a (green) coloured web edge as diagrammatic representative for the associated 1-morphism in the local impotent completion $\kdefFoam$:

\[
\begin{tikzpicture}[line cap = round, scale=.5, anchorbase]
    %webs and labled points
    \draw[very thick, cl_green] (0,0) to (2,0);
    \node at (0,0)[opacity=1,circle,fill, inner sep=1pt](){};
    \node at (2,0)[opacity=1,circle,fill, inner sep=1pt](){};
\end{tikzpicture} 
\;\;:=\;\;
\left ( \begin{tikzpicture}[line cap = round, scale=.5]
    %webs and labled points
    \draw[very thick] (0,0) to (2,0);
    \node at (0,0)[opacity=1,circle,fill, inner sep=1pt](){};
    \node at (2,0)[opacity=1,circle,fill, inner sep=1pt](){};
\end{tikzpicture} , \;\begin{tikzpicture}[anchorbase,line cap = round, scale=.5]
    %shading
	\path[fill=cl_green, opacity = 0.2] (0,0) to (0,2) to (2,2) to (2,0);
    %webs and labled points
    \draw[very thick] (0,0) to (0,2);
    \draw[very thick] (0,2) to (2,2);
    \draw[very thick] (2,2) to (2,0);
    \draw[very thick] (2,0) to (0,0);
    \node at (0,0)[opacity=1,circle,fill, inner sep=1pt](){};
    \node at (0,2)[opacity=1,circle,fill, inner sep=1pt](){};
    \node at (2,2)[opacity=1,circle,fill, inner sep=1pt](){};
    \node at (2,0)[opacity=1,circle,fill, inner sep=1pt](){};
\end{tikzpicture} \right ) 
\;\;=\;\;
(\id_k, F_A) 
\quad 
\text{where} 
\quad 
F_A =  
\begin{tikzpicture}[anchorbase,line cap = round, scale=.5]
    %shading
	\path[fill=cl_green, opacity = 0.2] (0,0) to (0,2) to (2,2) to (2,0);
    %webs and labled points
    \draw[very thick] (0,0) to (0,2);
    \draw[very thick] (0,2) to (2,2);
    \draw[very thick] (2,2) to (2,0);
    \draw[very thick] (2,0) to (0,0);
    \node at (0,0)[opacity=1,circle,fill, inner sep=1pt](){};
    \node at (0,2)[opacity=1,circle,fill, inner sep=1pt](){};
    \node at (2,2)[opacity=1,circle,fill, inner sep=1pt](){};
    \node at (2,0)[opacity=1,circle,fill, inner sep=1pt](){};
\end{tikzpicture} 
:=
\begin{tikzpicture}[anchorbase,line cap = round, scale=.5]
    %shading
	\path (0,0) to (0,2) to (2,2) to (2,0);
    %webs and labled points
    \draw[very thick] (0,0) to (0,2);
    \draw[very thick] (0,2) to (2,2);
    \draw[very thick] (2,2) to (2,0);
    \draw[very thick] (2,0) to (0,0);
    \node at (0,0)[opacity=1,circle,fill, inner sep=1pt](){};
    \node at (0,2)[opacity=1,circle,fill, inner sep=1pt](){};
    \node at (2,2)[opacity=1,circle,fill, inner sep=1pt](){};
    \node at (2,0)[opacity=1,circle,fill, inner sep=1pt](){};
     \node at (1,1){\tiny{$\idem_A$}}; 
\end{tikzpicture} 
 \]
In this diagrammatic language we colour the endpoints of the edge black to emphasize that the source and target of the resulting 1-morphism are still $(k)$.

More generally, given a web $W$ representing a 1-morphism in $\defNFoam$, we can choose one indecomposable idempotent for each edge and use it to decorate the identity foam on $W$ to obtain an idempotent 2-morphism on $W$. We use colours to encode these idempotent labels on foam facets and, by extension, to the the edges of $W$ to describe the corresponding 1-morphisms in $\kdefFoam$. For instance:
\begin{equation}
    \label{eq:trivalentcolored}
    \begin{tikzpicture}[anchorbase, line cap = round, scale=.4]
		\draw [very thick, cl_green] (2.25,0) to (.75,0);
		\draw [very thick, cl_red] (.75,0) to [out=135,in=0] (-1,.75);
		\draw [very thick, cl_dark_blue] (.75,0) to [out=225,in=0] (-1,-.75);
        \node at (2.25,0)[circle,fill,inner sep=1pt](k){};
        \node at (-1,.75)[circle,fill,inner sep=1pt](l){};
        \node at (-1,-.75)[circle,fill,inner sep=1pt](m){};
	\end{tikzpicture}
\;\;:=\;\;
    \left( 
    \begin{tikzpicture}[anchorbase, line cap = round, scale=.4]
		\draw [very thick] (2.25,0) to (.75,0);
		\draw [very thick] (.75,0) to [out=135,in=0] (-1,.75);
		\draw [very thick] (.75,0) to [out=225,in=0] (-1,-.75);
        \node at (2.25,0)[circle,fill,inner sep=1pt](k){};
        \node at (-1,.75)[circle,fill,inner sep=1pt](l){};
        \node at (-1,-.75)[circle,fill,inner sep=1pt](m){};
	\end{tikzpicture} ,  \;
    \begin{tikzpicture} [anchorbase, scale=.5]
	%shading
	\path[fill=cl_green, ,fill opacity=0.2] (2.25,3) to (.75,3) to (.75,0) to (2.25,0);\path[fill=cl_dark_blue, ,fill opacity=0.2] (.75,3) to [out=225,in=0] (-.5,2.5) to (-.5,-.5) to [out=0,in=225] (.75,0);
	\path[fill=cl_red, ,fill opacity=0.2] (.75,3) to [out=135,in=0] (-1,3.5) to (-1,.5) to [out=0,in=135] (.75,0);	
	%bottom web
	\draw [very thick, line cap = round] (2.25,0) to (.75,0);
	\draw [very thick, line cap = round] (.75,0) to [out=135,in=0] (-1,.5);
	\draw [very thick, line cap = round] (.75,0) to [out=225,in=0] (-.5,-.5);
    \node at (2.25,0)[circle,fill, fill opacity = 1,inner sep=1pt](k){};
    \node at (-1,.5)[circle,fill, fill opacity = 1,inner sep=1pt](l){};
    \node at (-.5,-.5)[circle,fill, fill opacity = 1, inner sep=1pt](m){};
	%seam
	\draw[very thick, line cap = round, dashed] (.75,0) to (.75,3);
	%vertical edges
	\draw [very thick] (2.25,3) to (2.25,0);
	\draw [very thick] (-1,3.5) to (-1,.5);
	\draw [very thick] (-.5,2.5) to (-.5,-.5);
	%top web
	\draw [very thick, line cap = round] (2.25,3) to (.75,3);
	\draw [very thick, line cap = round] (.75,3) to [out=135,in=0] (-1,3.5);
	\draw [very thick, line cap = round] (.75,3) to [out=225,in=0] (-.5,2.5);
    \node at (2.25,3)[circle,fill, fill opacity = 1,inner sep=1pt](k){};
    \node at (-1,3.5)[circle,fill, fill opacity = 1,inner sep=1pt](l){};
    \node at (-.5,2.5)[circle,fill, fill opacity = 1, inner sep=1pt](m){};
    \end{tikzpicture}
    \right) 
    \end{equation}
\end{conv}

\begin{proposition}[{\cite[Lemma 4.4, Lemma 4.6]{RW16}}]
    \label{prop:inv_foams}
    Let $A$ and $B$ be two disjoint multisubsets of $\Sigma$ such that $|A| = a$ and $|B| = b$ for $a, b \in \N_0$. Then, the following admissibly coloured foams are invertible as 2-morphisms in $\kdefFoam$, where the red facet is coloured by $\idem_A$, the blue by $\idem_B$ and the green by $\idem_{A \uplus B}$.
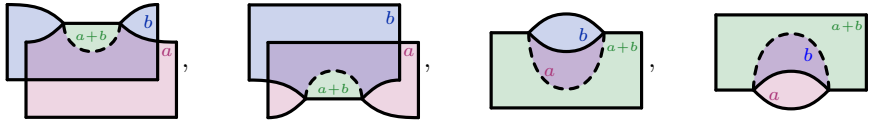
\[
    \begin{tikzpicture}[anchorbase, line cap = round, scale=.5,fill opacity=0.2]
	%shading	
	%	\node[opacity=1] at (.37,2) {$+$};
	%	\node[opacity=1] at (.13,5) {$+$};
	\path[fill=cl_dark_blue]  (2,3) to [out=180,in=45] (1,2.5) to [out=270,in=0] (0.25,1.75) to [out=180,in=270] (-0.5,2.5) to [out=135,in=0] (-2,3) to (-2,1) to (2,1) to (2,3);
	\path[fill=cl_red]  (2.5,2) to [out=180, in=315] (1,2.5)to [out=270,in=0] (0.25,1.75) to [out=180,in=270] (-0.5,2.5) to [out=225, in=0] (-1.5,2) to (-1.5,0) to (2.5,0) to (2.5,2);
	\path[fill=cl_green] (1,2.5) to [out=270, in=0] (0.25,1.75) to [out=180, in=270] (-0.5,2.5);
	% bottom web
	\draw[very thick] (2,1) to [out=180,in=0] (-2,1);
	\draw[very thick] (2.5,0) to [out=180,in=0] (-1.5,0);
	%\draw[very thick, directed=.55] (1,.75) to [out=225,in=45] (-.5,.25);
	%vertical edges
	\draw[very thick] (2,1) to (2,3);
	\draw[very thick] (2.5,0) to (2.5,2);
	\draw[very thick] (-1.5,0) to (-1.5,2);
	\draw[very thick] (-2,1) to (-2,3);	
	%middle web
	\draw[very thick] (2,3) to [out=180,in=45] (1,2.5);
	\draw[very thick] (2.5,2) to [out=180,in=315] (1,2.5);
	\draw[very thick] (1,2.5) to (-.5,2.5);
	\draw[very thick] (-.5,2.5) to [out=225,in=0] (-1.5,2);
	\draw[very thick] (-.5,2.5) to [out=135,in=0] (-2,3);
	%seams
	\draw[very thick, black, dashed] (1,2.5) to [out=270,in=0]  (0.25,1.75) to [out=180, in = 270] (-0.5,2.5);
	%labels
	\node[cl_dark_blue,opacity=1] at (1.75,2.55) {\tiny $b$};
	\node[cl_red,opacity=1] at (2.25,1.75) {\tiny $a$};
	\node[cl_green, opacity=1] at (.25,2.15) {\tiny $_{a+b}$};

    \end{tikzpicture},  \qquad
    \begin{tikzpicture} [anchorbase, line cap = round, scale=.5,fill opacity=0.2]
	%shading	
%	\node[opacity=1] at (.13,2) {$+$};
%	\node[opacity=1] at (.37,2) {$+$};
%% red
	\path[fill=cl_red] (2.5,0) to [out=180,in=315] (1,0.5) to [out=90,in=0]  (0.25,1.25) to [out=180, in=90] (-0.5,0.5) to [out=225, in=0] (-1.5,0) to (-1.5,2) to (2.5,2) to (2.5,0);
	\path[fill=cl_dark_blue] (2,1) to [out=180,in=45] (1,0.5) to [out=90,in=0]  (0.25,1.25) to [out=180, in=90] (-0.5,0.5) to [out=135, in=0] (-2,1) to (-2,3) to (2,3) to (2,1);
	\path[fill=cl_green]  (1,0.5) to [out=90,in=0]  (0.25,1.25) to [out=180, in=90] (-0.5,0.5) to (1,0.5);
%bottom web
	\draw[very thick] (2,1) to [out=180,in=45] (1,.5);
	\draw[very thick] (2.5,0) to [out=180,in=315] (1,.5);
	\draw[very thick] (1,.5) to (-.5,.5);
	\draw[very thick] (-.5,.5) to [out=225,in=0] (-1.5,0);
	\draw[very thick] (-.5,.5) to [out=135,in=0] (-2,1);
%seams
	\draw[very thick, black, dashed] (-.5,.5) to [out=90,in=180] (.25,1.25) to [out=0, in=90] (1,.5);
%vertical edges
	\draw[very thick] (2,1) to (2,3);
	\draw[very thick] (2.5,0) to (2.5,2);
	\draw[very thick] (-1.5,0) to (-1.5,2);
	\draw[very thick] (-2,1) to (-2,3);
% middle web
	\draw[very thick] (2,3) to (-2,3);
	\draw[very thick] (2.5,2) to  (-1.5,2);
%labels
	\node[cl_green, opacity=1] at (.25,.75) {\tiny $_{a+b}$};
	\node[cl_dark_blue,opacity=1] at (1.75,2.65) {\tiny $b$};
	\node[cl_red,opacity=1] at (2.25,1.75) {\tiny $a$};
    \end{tikzpicture}, \qquad
    \begin{tikzpicture} [anchorbase,line cap = round, scale=.5,fill opacity=0.2]
		%shading	
%	\node[opacity=1] at (.13,2) {$+$};
%	\node[opacity=1] at (.37,2) {$+$};
%% red
    \path[fill=cl_dark_blue]  (1,2.5) to [out=135,in=0] (0,3) to [out=180, in=45] (-1,2.5) to [out=270,in=180] (0,1) to [out=0,in=270] (1,2.5) ;
    \path[fill=cl_red]  (1,2.5) to [out=225,in=0] (0,2) to [out=180, in=315] (-1,2.5) to [out=270,in=180] (0,1) to [out=0,in=270] (1,2.5) ;
    \path[fill=cl_green]   (-1,2.5) to [out=270,in=180] (0,1) to [out=0,in=270] (1,2.5) to (2, 2.5) to (2,0.5) to (-2 , 0.5) to (-2, 2.5) to (-1,2.5);
%bottom web
	\draw[very thick] (2,0.5) to (-2,0.5);
%seams
	\draw[very thick, black, dashed] (-1,2.5) to [out=270,in=180] (0,1) to [out=0,in=270] (1,2.5);
%vertical edges
	\draw[very thick] (2,0.5) to (2,2.5);
	\draw[very thick] (-2,0.5) to (-2,2.5);
    \draw[very thick] (2,2.5) to (1,2.5);
	\draw[very thick] (-1,2.5) to  (-2,2.5);
	\draw[very thick] (1,2.5) to [out=135,in=0] (0,3) to [out=180, in=45] (-1,2.5);
	\draw[very thick] (1,2.5) to [out=225,in=0] (0,2) to [out=180,in=315] (-1,2.5);
%labels
	\node[cl_dark_blue,opacity=1] at (0.45,2.45) {\tiny $b$};
	\node[cl_green, opacity=1] at (1.5,2.1) {\tiny $_{a+b}$};
	\node[cl_red,opacity=1] at (-0.45,1.45) {\tiny $a$};

    \end{tikzpicture}, \qquad
    \begin{tikzpicture} [anchorbase,line cap = round, scale=.5,fill opacity=0.2]
		%shading	
%	\node[opacity=1] at (.13,2) {$+$};
%	\node[opacity=1] at (.37,2) {$+$};
%% red
    \path[fill=cl_dark_blue]  (-1,.5) to [out=90,in=180] (0,2) to [out=0, in=90] (1,.5) to [out=135,in=0] (0,1) to [out=180,in=45] (-1,.5) ;
    \path[fill=cl_red]  (1,0.5) to [out=225,in=0] (0,0) to [out=180, in=315] (-1,.5) to [out=90,in=180] (0,2) to [out=0,in=90] (1,.5);
	\path[fill=cl_green]  (2,0.5) to (2,2.5) to (-2, 2.5) to (-2,0.5) to (-1,0.5) to [out=90,in=180] (0,2) to [out=0,in=90] (1,0.5) to (2,0.5);
%bottom web
	\draw[very thick] (2,0.5) to (1,0.5);
	\draw[very thick] (-1,0.5) to  (-2,0.5);
	\draw[very thick] (1,0.5) to [out=135,in=0] (0,1) to [out=180, in=45] (-1,0.5);
	\draw[very thick] (1,0.5) to [out=225,in=0] (0,0) to [out=180,in=315] (-1,0.5);
%seams
	\draw[very thick, black, dashed] (1,.5) to [out=90,in=0] (0,2) to [out=180,in=90] (-1,.5);
%vertical edges
	\draw[very thick] (2,0.5) to (2,2.5);
	\draw[very thick] (-2,0.5) to (-2,2.5);
%middle web
		\draw[very thick] (2,2.5) to (-2,2.5);
%labels
	\node[blue,opacity=1] at (0.45,1.45) {\tiny $b$};
	\node[cl_green, opacity=1] at (1.5,2.2) {\tiny $_{a+b}$};
	\node[cl_red,opacity=1] at (-0.45,0.35) {\tiny $a$};
    \end{tikzpicture}
    \]

\end{proposition}

\subsection{Graphical language for 2-categorical idempotent completion}

Using the 2-categorical idempotent completion, we can also decompose the objects of $\defNFoam$.

\begin{theorem}
\label{thm:object_coloring}
Let $\Sigma$ be as in \Cref{conv:sigma} and $0\leq k \leq N=|\Sigma|$. Then we have the following canonical equivalence in $\kkdefFoam$  
 \begin{equation}\label{eq:objsumdecomp}
     k \simeq \bigboxplus_{\substack{A\in \powerset(\Sigma)\\ |A|=k }} \kartwobj{k}{A}
 \end{equation}
compatible with the inclusions $H_A^\Sigma \hookrightarrow H_k^\Sigma$ from  \eqref{eq:Hsumdecomp} so that the isomorphism from \Cref{lem:End-of-edge} restricts to $\End(\id_{\kartwobj{k}{A}})\cong H_A^\Sigma$ for all $A\in \powerset(\Sigma)$ with $|A|=k$. 
\end{theorem}
\begin{proof}
This follows directly from \Cref{prop:idem_sumtwo}. Namely, the 2-categorical idempotent \[(k, (\id_k, F_A), F_A, F_A),\] where $F_A$ denotes the identity foam coloured by the idempotent $\idem_A$ on a thickness $k$ web, splits via a manifestly split 2-categorical idempotent $(k, \kartwobj{k}{A}, f, g, \phi, \gamma)$. 
\end{proof}

\begin{corollary}
    Let $\mathbf{a} = (a_1 , \ldots , a_l)$ be a sequence of uncoloured points and $B_i \subset \Sigma$ multisubsets such that $|B_i| = a_i$ for $1 \leq i \leq l$. Then the sequence $\mathbf{a}$ splits via a sequence of coloured points $(\kartwobj{a_1}{B_1} , \ldots , \kartwobj{a_l}{B_l})$.
\end{corollary}

\begin{proof}
    Follows from \Cref{lem:tensor_end_ring} and \Cref{thm:object_coloring}. 
\end{proof}

\begin{conv}[Graphical language in the 2-categorical idempotent completion]
Recall the graphical language for the 1-categorical idempotent completion of $\defNFoam$ from \Cref{conv:col_web} and, in particular, for an object $(k)$ the idempotent foam $F_A$ on $\id_k$. \Cref{thm:object_coloring} shows that the object $(k)$ with the 1-morphism $(\id_k,F_A)$ in $\kdefFoam$ admits the structure of a 2-categorical idempotent $(k, (\id_k, F_A), F_A, F_A)$. We use the following diagrammatic representative for the associated object $\kartwobj{k}{A}$ in the 2-categorical idempotent completion $\kkdefFoam$:

\[ \begin{tikzpicture}[anchorbase, line cap = round, scale=.4]
\node at (2.25,0)[circle,cl_green,fill,inner sep=1pt,label=right:$\kartwobj{k}{A}$](kA){};
\end{tikzpicture} 
:=
\left(
\begin{tikzpicture}[anchorbase, line cap = round, scale=.4]
\node at (2.25,0)[circle,fill,inner sep=1pt,label=right:$k$](kA){};
\end{tikzpicture}
,\;
\begin{tikzpicture}[line cap = round, scale=.5, anchorbase]
    %webs and labled points
    \draw[very thick, cl_green] (0,0) to (2,0);
    \node at (0,0)[opacity=1,circle,fill, inner sep=1pt](){};
    \node at (2,0)[opacity=1,circle,fill, inner sep=1pt](){};
\end{tikzpicture}
,\;
\begin{tikzpicture}[anchorbase,line cap = round, scale=.5]
    %shading
	\path[fill=cl_green, opacity = 0.2] (0,0) to (0,2) to (2,2) to (2,0);
    %webs and labled points
    \draw[very thick] (0,0) to (0,2);
    \draw[very thick] (0,2) to (2,2);
    \draw[very thick] (2,2) to (2,0);
    \draw[very thick] (2,0) to (0,0);
    \node at (0,0)[opacity=1,circle,fill, inner sep=1pt](){};
    \node at (0,2)[opacity=1,circle,fill, inner sep=1pt](){};
    \node at (2,2)[opacity=1,circle,fill, inner sep=1pt](){};
    \node at (2,0)[opacity=1,circle,fill, inner sep=1pt](){};
\end{tikzpicture}
\right) \]
Note that the \emph{coloured point} now carries the same colouring as the sheet. 

From the data witnessing the splitting of $k$ via $\kartwobj{k}{A}$ in the proof of \Cref{thm:object_coloring}, we get two 1-morphisms $f_{\kartwobj{k}{A}} \colon \kartwobj{k}{A} \ot k$ and  $g_{\kartwobj{k}{A}} \colon k \ot \kartwobj{k}{A}$ diagrammatically represented by:
\begin{align*}
      \begin{tikzpicture}[anchorbase, line cap = round, scale=.5] 
      \draw [ultra thick, dotted, cl_green] (0,0) to (2,0);
      \node at (2,0)[circle,fill,inner sep=1.5pt,label=right:$k$](k){};
      \node at (0,0)[circle,fill,cl_green,inner sep=1.5pt,label=left:$\kartwobj{k}{A}$](k){};
      \end{tikzpicture} 
      \qquad
     \begin{tikzpicture}[anchorbase, line cap = round, scale=.5] 
      \draw [ultra thick, dotted, cl_green] (0,0) to (2,0);
      \node at (2,0)[circle,fill, cl_green ,inner sep=1.5pt,label=right:$\kartwobj{k}{A}$](k){};
      \node at (0,0)[circle,fill,inner sep=1.5pt,label=left:$k$](k){};
      \end{tikzpicture}
\end{align*}
We also get the following 2-morphisms $\phi_A \colon f_{\kartwobj{k}{A}} \circ  g_{\kartwobj{k}{A}} \Rightarrow \id_{\kartwobj{k}{A}}$ and $\gamma_A \colon \id_{\kartwobj{k}{A}} \Rightarrow f_{\kartwobj{k}{A}} \circ  g_{\kartwobj{k}{A}}$
which are actually mutual inverses by \Cref{prop:idem_sum}, yielding the 2-isomorphism 
\begin{align}
\label{eq:green_identity}
\begin{tikzpicture}[anchorbase, line cap = round, scale=.4]
        \draw [ultra thick, dotted, cl_green] (0,0) to (4,0);
        \node at (2,0)[circle,fill, inner sep=1.5pt](m){};
        \node at (0,0)[circle,fill, cl_green, inner sep=1.5pt](m){};
        \node at (4,0)[circle,fill,cl_green,inner sep=1.5pt](m){};
    \end{tikzpicture} \cong 
    \begin{tikzpicture}[anchorbase, line cap = round, scale=.4]
        \draw [ultra thick, cl_green] (0,0) to (4,0);
        \node at (0,0)[circle,fill,cl_green, inner sep=1.5pt](m){};
        \node at (4,0)[circle,fill,cl_green, inner sep=1.5pt](m){};
\end{tikzpicture}    
\end{align} 
where the right hand side denotes the identity 1-morphism of $\kartwobj{k}{A}$. Additionally, the data of the splitting yields the 2-isomorphism:
\begin{align}
\label{eq:green_idempotent}
    \begin{tikzpicture}[anchorbase, line cap = round, scale=.4]
        \draw [ultra thick, dotted, cl_green] (0,0) to (4,0);
        \node at (2,0)[circle,fill, cl_green, inner sep=1.5pt](m){};
        \node at (0,0)[circle,fill, inner sep=1.5pt](m){};
        \node at (4,0)[circle,fill,inner sep=1.5pt](m){};
    \end{tikzpicture} \cong 
    \begin{tikzpicture}[anchorbase, line cap = round, scale=.4]
        \draw [ultra thick, cl_green] (0,0) to (4,0);
        \node at (0,0)[circle,fill, inner sep=1.5pt](m){};
        \node at (4,0)[circle,fill, inner sep=1.5pt](m){};
\end{tikzpicture}.
\end{align} 
All these isomorphisms are represented by the coloured sheet $F_A$.
\smallskip

\noindent Next we extend this graphical calculus to 1-morphisms in $\kkdefFoam$. Consider a coloured web as in \Cref{conv:col_web}, representing a 1-morphism in $\kdefFoam$, e.g. \eqref{eq:trivalentcolored}.
In $\kkdefFoam$ we can now pre- and post-compose with the corresponding inclusion and projection 1-morphisms to obtain an associated 1-morphism between the split objects, e.g.:

\[
       \begin{tikzpicture}[anchorbase, line cap = round, scale=.4]
		\draw [very thick, cl_green] (2.25,0) to (.75,0);
		\draw [very thick, cl_red] (.75,0) to [out=135,in=0] (-1,.75);
		\draw [very thick, cl_dark_blue] (.75,0) to [out=225,in=0] (-1,-.75);
        \node at (2.25,0)[circle,fill, cl_green,inner sep=1pt](k){};
        \node at (-1,.75)[circle,fill,cl_red,inner sep=1pt](l){};
        \node at (-1,-.75)[circle,fill,cl_dark_blue,inner sep=1pt](m){};
	\end{tikzpicture} 
        \;\; :=\;\;
\begin{tikzpicture}[anchorbase, line cap = round, scale=.4]
		\draw [very thick, cl_green] (2.25,0) to (.75,0);
        \draw [very thick, dotted, cl_green] (2.25,0) to (3.75,0);
		\draw [very thick, cl_red] (.75,0) to [out=135,in=0] (-1,.75);
        \draw [very thick, dotted, cl_green] (2.25,0) to (3.75,0);
		\draw [very thick, cl_dark_blue] (.75,0) to [out=225,in=0] (-1,-.75);
        \draw [very thick, dotted, cl_dark_blue] (-1,-.75) to (-2.5,-.75);
        \draw [very thick, dotted, cl_red] (-1,.75) to (-2.5,.75);
        %\node at (2.25,0)[circle,fill,inner sep=1pt,label=above:$k$](k){}
        \node at (2.25,0)[circle,fill,inner sep=1pt](k){};
        % \node at (3.75,0)[circle,fill, cl_green ,inner sep=1pt,label=right:$\kartwobj{k}{A}$](k){};
        % \node at (-2.5,.75)[circle,fill, cl_red, inner sep=1pt,label=left:$\kartwobj{l}{B}$](l){};
        % \node at (-2.5,-.75)[circle,fill,cl_dark_blue, inner sep=1pt,label=left:$\kartwobj{m}{C}$](m){};
        \node at (3.75,0)[circle,fill, cl_green ,inner sep=1pt](k){};
        \node at (-2.5,.75)[circle,fill, cl_red, inner sep=1pt](l){};
        \node at (-2.5,-.75)[circle,fill,cl_dark_blue, inner sep=1pt](m){};
        \node at (-1,.75)[circle,fill,inner sep=1pt](l){};
        \node at (-1,-.75)[circle,fill,inner sep=1pt](m){};
	\end{tikzpicture} 
    \]  
    In this example, with the colour coding and notation of \Cref{prop:inv_foams}, the 1-morphism in $\defNFoam$  represented by the splitting web $(a,b)\ot (c)$ gives rise to a 1-morphism $(a[A],b[B])\ot (c[C])$ in $\kkdefFoam$. 
\end{conv}

\newcommand{\colweb}{
\draw [very thick, green] (3,0) to (1.5,0);
        \draw [very thick, cl_red] (1.5,0) to [out=135,in=0] (0,.75);
		\draw [very thick, cl_dark_blue] (1.5,0) to [out=225,in=0] (0,-.75);
}

\newcommand{\colwebtwo}{
\colweb
        \node at (0,.75)[circle,fill, cl_red, inner sep=1pt](l){};
        \node at (0,-.75)[circle,fill,cl_dark_blue, inner sep=1pt](m){};}

\newcommand{\colwebone}{
\colweb
        \node at (3,0)[circle,fill,cl_green, inner sep=1pt](m){};
}

\begin{proposition}
    \label{prop:web_isos}
    Let $A$ and $B$ be two disjoint multisubsets of $\Sigma$ such that $|A| = a$ and $|B| = b$ for $a, b \in \N_0$. We encode idempotent labelings by colours: red for $\idem_A$, blue for $\idem_B$ and green for $\idem_{A \uplus B}$. Then in $\kkdefFoam$ the following 1-morphisms are canonically isomorphic:

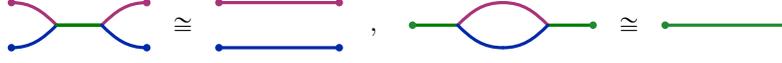
\[
\begin{tikzpicture}[anchorbase, line cap = round, scale=.4]
		\colwebtwo 
        \begin{scope}[xscale=-1,shift={(-4.5,0)}]
        \colwebtwo
        \end{scope}
        \end{tikzpicture}
        \;\;\cong\;\;
        \begin{tikzpicture}[anchorbase, line cap = round, scale=.4]
		\draw [very thick, cl_red] (0,0) to (4, 0);
        \draw [very thick, cl_dark_blue] (0,-1.5) to (4, -1.5);
        \node at (0,-1.5)[circle,fill,cl_dark_blue, inner sep=1pt](){};
        \node at (4, -1.5)[circle,fill,cl_dark_blue, inner sep=1pt](){};
        \node at (0,0)[circle,fill,cl_red, inner sep=1pt](){};
        \node at (4,0)[circle,fill,cl_red, inner sep=1pt](){};
	\end{tikzpicture}
    \quad,\quad
    \begin{tikzpicture}[anchorbase, line cap = round, scale=.4]
		\colwebone 
        \begin{scope}[xscale=-1,shift={(0,0)}]
        \colwebone
        \end{scope}
        \end{tikzpicture}
        \;\;\cong\;\;
           \begin{tikzpicture}[anchorbase, line cap = round, scale=.4]
            \draw [very thick, cl_green] (0,0) to (4, 0);
            \node at (0,0)[circle,fill,cl_green, inner sep=1pt](){};
            \node at (4,0)[circle,fill,cl_green, inner sep=1pt](){};
        \end{tikzpicture}     
\]

\noindent    In particular, we get an equivalence $(\kartwobj{a}{A}, \kartwobj{b}{B}) \simeq (\kartwobj{(a+b)}{A\uplus B})$ in $\kkdefFoam$. 
\end{proposition}

\begin{proof}

    \[
    \begin{tikzpicture}[anchorbase, line cap = round, scale=.4]
		\colwebtwo 
        \begin{scope}[xscale=-1,shift={(-4.5,0)}]
        \colwebtwo
        \end{scope}
        \end{tikzpicture}
        \;\;=\;\;
    \begin{tikzpicture}[anchorbase, line cap = round, scale=.4]
		\draw [very thick, cl_green] (2.25,0) to (.75,0);
       % \draw [very thick, dotted, cl_green] (2.25,0) to (5.25,0);
		\draw [very thick, cl_red] (.75,0) to [out=135,in=0] (-1,.75);
		\draw [very thick, cl_dark_blue] (.75,0) to [out=225,in=0] (-1,-.75);
        \draw [very thick, dotted, cl_dark_blue] (-1,-.75) to (-2.5,-.75);
        \draw [very thick, dotted, cl_red] (-1,.75) to (-2.5,.75);
        \draw [very thick, cl_green] (2.25,0) to (4,0);
        \draw [very thick, cl_dark_blue] (4,0) to [out=315,in=180] (5.75,-.75);
        \draw [very thick, cl_red] (4,0) to [out=45,in=180] (5.75,.75);
        \draw [very thick, dotted, cl_dark_blue] (5.75,-.75) to (7.25,-.75);
        \draw [very thick, dotted, cl_red] (5.75,.75) to (7.25,.75);
        %\node at (2.25,0)[circle,fill,inner sep=1pt,label=above:$k$](k){}
        % \node at (2.25,0)[circle,fill,inner sep=1pt](k){};
        % \node at (5.25,0)[circle,fill,inner sep=1pt](k){};
        %\node at (3.75,0)[circle,fill, cl_green ,inner sep=1pt](k){};
        \node at (-2.5,.75)[circle,fill, cl_red, inner sep=1pt](l){};
        \node at (-2.5,-.75)[circle,fill,cl_dark_blue, inner sep=1pt](m){};
        \node at (7.25,.75)[circle,fill, cl_red, inner sep=1pt](l){};
        \node at (7.25,-.75)[circle,fill,cl_dark_blue, inner sep=1pt](m){};
        \node at (-1,.75)[circle,fill,inner sep=1pt](l){};
        \node at (-1,-.75)[circle,fill,inner sep=1pt](m){};
        \node at (5.75,-.75)[circle,fill,inner sep=1pt](k){};
        \node at (5.75,.75)[circle,fill,inner sep=1pt](k){};
	\end{tikzpicture} 
 \;\;   \cong \;\;
    \begin{tikzpicture}[anchorbase, line cap = round, scale=.4]
				\draw [very thick, cl_red] (0,0) to (4, 0);
        \draw [very thick, cl_dark_blue] (0,-1.5) to (4, -1.5);
        \draw [very thick, dotted, cl_dark_blue] (0,-1.5) to (-1.5,-1.5);
        \draw [very thick, dotted, cl_red] (0,0) to (-1.5,0);
        \draw [very thick, dotted, cl_dark_blue] (4, -1.5) to (5.5, -1.5);
        \draw [very thick, dotted, cl_red] (4,0) to (5.5,0);
        \node at (-1.5,-1.5)[circle,fill,cl_dark_blue, inner sep=1pt](m){};
        \node at (5.5, -1.5)[circle,fill,cl_dark_blue, inner sep=1pt](m){};
        \node at (-1.5,0)[circle,fill,cl_red, inner sep=1pt](m){};
        \node at (5.5,0)[circle,fill,cl_red, inner sep=1pt](m){};
        \node at (0,-1.5)[circle,fill,inner sep=1pt](l){};
        \node at (0,0)[circle,fill,inner sep=1pt](m){};
        \node at (4, -1.5)[circle,fill,inner sep=1pt](k){};
        \node at (4,0)[circle,fill,inner sep=1pt](k){};
	\end{tikzpicture} 
     \;\;=\;\;
     \begin{tikzpicture}[anchorbase, line cap = round, scale=.4]
		\draw [very thick, cl_red] (0,0) to (4, 0);
        \draw [very thick, cl_dark_blue] (0,-1.5) to (4, -1.5);
        \node at (0,-1.5)[circle,fill,cl_dark_blue, inner sep=1pt](){};
        \node at (4, -1.5)[circle,fill,cl_dark_blue, inner sep=1pt](){};
        \node at (0,0)[circle,fill,cl_red, inner sep=1pt](){};
        \node at (4,0)[circle,fill,cl_red, inner sep=1pt](){};
	\end{tikzpicture} 
    \] 
    where the non-definitional isomorphism comes from \Cref{prop:inv_foams}. The other isomorphism and equivalence are established analogously.
\end{proof}

\begin{definition}
\label{def:inverse_webs}

Let $A$ be a multisubset of $\Sigma$ with $|A|=k$ and recall the partition into $A_i := \{\lambda_i^{k_i}\}$ with $|A_i|=k_i$ from \Cref{conv:sigma}. We define 1-morphisms  \[ \bigboxtimes_i k_i[A_{i}]  \;\xleftarrow{s} \; k[A] \quad, \quad  \kartwobj{k}{A} \;\xleftarrow{r} \; \bigboxtimes_i \kartwobj{k_i}{A_i}\]
in $\kkdefFoam$ by splitting (resp. merging) webs, with strands coloured so as to split (resp. merge) the $A$-coloured $k$-labeled endpoint into $l$ $A_i$-coloured $k_i$-labeled endpoints for $1 \leq i \leq l$. For the sake of definiteness, we choose upward-adjusted trees with minimal possible numbers of vertices as underlying webs, see e.g.~\Cref{fig:splitting}.
\end{definition}

\begin{figure}[h]
    \centering
\[ s=  \begin{tikzpicture}[anchorbase, line cap = round, scale=.4]
	\draw [cl_dark_blue, very thick] (2.25,0) to (.75,0);
	\draw [cl_green, very thick] (.75,0) to [out=135,in=0] (-1,.75);
	\draw [cl_red, very thick] (.75,0) to [out=225,in=0] (-4.5,-1.55);
	\draw [orange, very thick] (-1,.75) to [out=135,in=0] (-2.75,1.55);
	\draw [gray, very thick] (-1,.75)to [out=225,in=0] (-4.5,-.25);
	\draw [cyan, very thick] (-2.75,1.55) to [out=135,in=0] (-4.5,2.35);
	\draw [cl_purple, very thick] (-2.75,1.55)to [out=225,in=0] (-4.5,1.05);

    %labels 
    \node at (-4.5,-1.55)[circle,fill, cl_red, inner sep=1pt, label=left:\tiny{$\kartwobj{k}{A_4}$}](4){};
    \node at (-4.5,-.25)[circle,fill, gray, inner sep=1pt, label=left:\tiny{$\kartwobj{k}{A_3}$}](3){};
    \node at (-4.5,1.05)[circle,fill, cl_purple, inner sep=1pt, label=left:\tiny{$\kartwobj{k}{A_2}$}](2){};
    \node at (-4.5,2.35)[circle,fill, cyan, inner sep=1pt, label=left:\tiny{$\kartwobj{k}{A_1}$}](1){};
    \node at (2.25,0)[circle,fill, cl_dark_blue, inner sep=1pt, label=right:\tiny{$\kartwobj{k}{A}$}](k){};
\end{tikzpicture} \qquad 
r = \begin{tikzpicture}[anchorbase, x={(-1,0)},y={(0,1)}, line cap = round, scale=.4]
	\draw [cl_dark_blue, very thick] (2.25,0) to (.75,0);
	\draw [cl_green, very thick] (.75,0) to [out=45,in=180] (-1,.75);
	\draw [cl_red, very thick] (.75,0) to [out=315
    ,in=180] (-4.5,-1.55);
	\draw [orange, very thick] (-1,.75) to [out=45,in=180] (-2.75,1.55);
	\draw [gray, very thick] (-1,.75)to [out=315,in=180] (-4.5,-.25);
	\draw [cl_blue, very thick] (-2.75,1.55) to [out=45,in=180] (-4.5,2.35);
	\draw [cl_purple, very thick] (-2.75,1.55)to [out=315,in=180] (-4.5,1.05);

    %labels 
    \node at (-4.5,-1.55)[circle,fill, cl_red, inner sep=1pt, label=right:\tiny{$\kartwobj{k}{A_4}$}](4){};
    \node at (-4.5,-.25)[circle,fill, gray, inner sep=1pt, label=right:\tiny{$\kartwobj{k}{A_3}$}](3){};
    \node at (-4.5,1.05)[circle,fill, cl_purple, inner sep=1pt, label=right:\tiny{$\kartwobj{k}{A_2}$}](2){};
    \node at (-4.5,2.35)[circle,fill, cl_blue, inner sep=1pt, label=right:\tiny{$\kartwobj{k}{A_1}$}](1){};
    \node at (2.25,0)[circle,fill, cl_dark_blue, inner sep=1pt, label=left:\tiny{$\kartwobj{k}{A}$}](k){};
\end{tikzpicture} \]   
    \caption{Morphisms $s$ and $r$ for $A = \{\lambda_{1}^{k_1},\lambda_{2}^{k_2}, \lambda_{3}^{k_3}, \lambda_{4}^{k_4} \}$}
    \label{fig:splitting}
\end{figure}
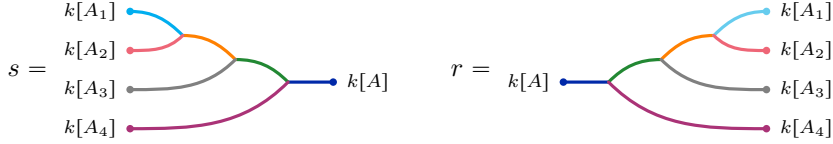

\begin{theorem}
\label{thm:full_split}
Retaining notation from \Cref{def:inverse_webs}, the 1-morphisms $s$ and $r$ furnish an equivalence $\kartwobj{k}{A} \simeq\bigboxtimes \kartwobj{k_i}{A_i}$ in $\kkdefFoam$. Together with \eqref{eq:objsumdecomp}, we obtain the decomposition 
    \begin{equation}
     \label{eq:objsumtensdecomp}
     k \simeq \bigboxplus_{\substack{A\in \powerset(\Sigma)\\ |A|=k }} \bigboxtimes_{i}^{l} \kartwobj{k_i}{A_i}.
 \end{equation}
\end{theorem}

\begin{proof}
Applying \Cref{prop:web_isos} inductively shows $s \circ r \cong \id_{\bigboxtimes \kartwobj{k_i}{A_i}}$ and $r \circ s \cong \id_{\kartwobj{k}{A}}$. 
\end{proof}

We can now extend the description of Lee homology from \cite{BNM06}. 

\begin{example} %lee homology as an example of what is said above
Consider the setting of Lee's deformed $\gltwo$ link homology, namely $\mathbf{2Foam}_+^\Sigma$ for $\Sigma = \{-1, 1\}$, \cite{W25}. The object $1 \in\mathbf{2Foam}_+^\Sigma$ decomposes by \Cref{thm:directsum} in $\kkar \kar\mathbf{2Foam}_+^\Sigma$ as 
\[1 \simeq \kartwobj{1}{\{1\}} \boxplus \kartwobj{1}{\{-1\}} \simeq \textcolor{cl_green}{1} \boxplus \textcolor{cl_red}{1},\]  
using the colour conventions from \cite[Section 4]{BNM06}. More generally, one can take $\Sigma = \{\mu, \lambda\}$  for $\lambda -\mu \in \C$ and obtain the splitting $1 \simeq \kartwobj{1}{\{\mu\}} \boxplus \kartwobj{1}{\{\lambda\}}$. We get a decomposition along the $\gltwo$ branching rule. 
\end{example}

For $\lambda_i\in \Sigma$ we consider the \emph{$\lambda_i$-monochromatic} full sub-2-category of $\kdefFoam$ on all objects and on all 1-morphisms represented by webs whose edges are coloured by idempotents indexed by multisubsets of $\Sigma$ that contain only $\lambda_i$, see \cite[Definition 5.15]{RW16}. By admissibility, all 2-morphisms can then be represented as linear combinations of foams whose facets likewise carry a $\lambda_i$-monochromatic idempotent colouring. As a consequence of \cite[Proposition 5.17]{RW16}, this 2-category can be identified with $\kar N_i\mathbf{Foam}_+^{\Sigma_i}$ where we use the notation $\Sigma_i = \{\lambda_i^{N_{i}} \} \subset \Sigma$ from \Cref{conv:sigma}. Analogously, we can consider $\kkar \kar N_i\mathbf{Foam}_+^{\Sigma_i}$ as full sub-2-category of $\kkdefFoam$. We now define and abusively denote by
\[
\tilde{\bigboxtimes}_{i} \kkar \kar N_i\mathbf{Foam}_+^{\Sigma_i}
\]
the sub-2-category of $\kkdefFoam$ on objects, 1-morphisms and 2-morphisms that are obtained by $\boxtimes$-tensoring\footnote{Strictly speaking, we have not defined $\boxtimes$ as a monoidal structure on $\kkdefFoam$, but it is clear how it would act on objects, 1-morphisms and 2-morphisms. Alternatively, we can equivalently construct a corresponding sub-2-category on the level of the local idempotent completions $\kar$, where the monoidal structure is inherited by \cite[Remark 3.17]{SW24}, and then proceed to $\kkar$.} objects, 1-morphisms and 2-morphisms from $\kkar \kar N_i\mathbf{Foam}_+^{\Sigma_i}$ in order specified by $1\leq i\leq l$. 

\begin{corollary}
    \label{cor:cat_split}
    The inclusion \[ \tilde{\bigboxtimes}_{i} \kkar \kar N_i\mathbf{Foam}_+^{\Sigma_i} \hookrightarrow  \kkdefFoam \] induces an equivalence of 2-categories. 
\end{corollary}

\begin{proof}
    Follows from the equivalence of objects from \Cref{thm:full_split} and \cite[\S 4.2,\S 4.3]{RW16}. 
\end{proof}

\begin{remark}
We take \Cref{cor:cat_split} and the results of \cite[\S 4.2,\S 4.3]{RW16} as evidence for an equivalence 
\[
\bigotimes_{i} \kkar \kar N_i\mathbf{Foam}_+^{\Sigma_i} \xrightarrow{\simeq}  \kkdefFoam
\]
as 2-idempotent complete monoidal 2-categories, where $\bigotimes$ now denotes a conjectural symmetric monoidal structure on a suitably defined ambient 4-category of monoidal 2-categories. 

Likewise, we expect that braided monoidal 2-categories\footnote{Or, more precisely, $\mathbb{E}_2$-monoidal $(\infty,2)$-categories.} obtained by taking pre-triangulated hulls locally decompose into an analogous way. This would give a fully local description of decomposition phenomena for deformed link homologies and their associated skein modules \cite{MWW24,RW24}. 
\end{remark}

\renewcommand*{\bibfont}{\small}
\setlength{\bibitemsep}{0pt}
\raggedright
\printbibliography

\end{document}